\documentclass[11pt]{article}
\usepackage{amsfonts}
\usepackage{color,xcolor}
\usepackage{amsmath}
\usepackage{latexsym,amsfonts,amssymb}
\usepackage{mathrsfs}
 \usepackage[pagewise]{lineno}%\linenumbers

\usepackage{amsthm,cite}

\newtheorem{thm}{\bf Theorem}[section]
\newtheorem{lem}[thm]{\bf Lemma}
\newtheorem{prop}[thm]{\bf Proposition}
\newtheorem{cor}[thm]{\bf Corollary}
\usepackage[hidelinks,colorlinks=true,linkcolor=black,citecolor=black]{hyperref} 
\newtheoremstyle{REMARK}{3pt}{3pt}{\rm}{}{\bf}{.}{0.5em}{}
\theoremstyle{REMARK}\newtheorem{rem}[thm]{Remark}
\newtheorem*{rem*}{Remark} \newtheorem*{ac*}{Acknowledgments}

\usepackage{CJKutf8}

\begin{document}
\begin{CJK}{UTF8}{gkai}

\title{Analysis of a nonlinear free boundary problem modeling the radial growth
 of two-layer tumors}

\author{Junde Wu$^\dagger$, \;\;Hao Xu$^\dagger$ and
   Yuehong Zhuang$^\ddagger$\footnote{Corresponding author. \newline
   \qquad E-mail: wujund@suda.edu.cn (J. Wu),    
   1158637462@qq.com (Hao Xu), zdzyh@outlook.com(Yuehong Zhuang)}
}

\date{$^\dagger$Department of Mathematics, Soochow University,
  \\ [0.1 cm]
  Suzhou,   Jiangsu 215006, PR China,
  \\ [0.2 cm]
  $^\ddagger$Department of Mathematics, Jinan University,
  \\ [0.1 cm]
  Guangzhou, Guangdong 510632, PR China.
  }
\medskip

\maketitle
\begin{abstract} 
In this paper we study a nonlinear free boundary problem on the radial growth of a two-layer solid tumor with a quiescent core. The tumor surface and its inner interface separating the proliferating cells and the quiescent cells are both free boundaries. By deeply analyzing their relationship and employing the maximum principle, we show this problem is globally well-posed and prove the existence of a unique positive threshold $\sigma^*$ such that the problem admits a unique stationary solution with a quiescent core if and only if the externally supplied nutrient $\bar\sigma> \sigma^*$. The stationary solution is globally asymptotically stable. The formation of the quiescent core and its interesting connection with the necrotic core are also given. 

\smallskip

{\bf 2020 Mathematics Subject Classification}:  35B40;  35R35;  35Q92

\smallskip

{\bf Keywords}:  free boundary problem; stationary solution;
two-layer tumor; stability

\medskip

\end{abstract}

\section{Introduction}
\setcounter{equation}{0}
\hskip 1em

%\newpage 

In this paper we study a free boundary problem modeling the  growth of a radially symmetric tumor with a layer of proliferating cells surrounding a central core full of quiescent cells. The model reads as follows:
\begin{equation}\label{1.1}
\left\{
\begin{array}{l}
  \Delta_r \sigma = f(\sigma)H(\sigma-\sigma_Q)+h(\sigma)
  \big(1-H(\sigma-\sigma_Q)\big)\qquad\mbox{for}\;\; 0<r<R(t),\;\;t>0,
\\ [0.3 cm]
  \sigma_r(0,t)=0,\quad \sigma(R(t),t)=\bar\sigma \qquad \mbox{for}\;\;t>0,
\\[0.3 cm]
  \displaystyle
  R^2(t){\frac{dR(t)}{dt}}=\int_{\sigma(r,t)>\sigma_Q} g(\sigma(r,t))r^2dr
  -\int_{\sigma(r,t)\le\sigma_Q}\nu r^2dr \qquad\mbox{for}\;\;t>0,
\\ [0.3 cm]
  R(0)=R_0,
\end{array}
\right.
\end{equation}
%(1.1)
  where $\sigma=\sigma(r, t)$ is the unknown concentration of the nutrient distributing in the tumor region and $R(t)$ is the unknown tumor radius over time $t>0$, with initial radius given by $R_0>0$.
  The radial Laplacian in $\mathbb{R}^3$ is denoted by $\Delta_r={\frac{d^2}{dr^2}}+{\frac2r}{\frac{d}{dr}}$ with $r=|x|$ for $x\in\mathbb{R}^3$, and the Heaviside function $H(\cdot)$ is defined by $H(s)=1$ for $s>0$ and $H(s)=0$ for $s\le 0$.
  In this model the tumor is supposed to have two layers at some stage of tumor evolution, and we  assume to distinguish the different cell layers of the tumor tissue by a critical constant $\sigma_Q>0$. This constant represents a threshold value of nutrient concentration for the tumor cells transitioning between the proliferating phase and the quiescent phase. 
  In this setting, the region $\{\sigma(r, t)>\sigma_Q\}$ is occupied by the proliferating cells, and the region $\{\sigma(r, t)\le\sigma_Q\}$ consists of all quiescent cells which, in a certain time span, may form a quiescent core. 
  The tumor is supplied by the nutrients from the external micro-environment via its surface at a constant level $\bar\sigma>0$, which gives $\sigma=\bar\sigma$ on $r=R(t)$. 
  The first equation of \eqref{1.1} accounts for the diffusion of nutrients and their net consumption by proliferating cells and quiescent cells at the rates of
  $f(\sigma)$ and $h(\sigma)$, respectively, where $f$ and $h$ are given functions satisfying some biological assumptions. 
  The third line of \eqref{1.1} stands for the evolution of the tumor's volume over time, where $g$ is another given function, which is usually referred to as \textit{proliferation rate function} for the tumor cells, and $\nu>0$ is the constant removal rate of quiescent cells. Typically, $f$, $h$ and $g$ are given by linear functions, cf. \cite{byr-cha-96,cui-fri-01,fri-rei-99},  with the following form:
\begin{equation}\label{1.2}
f(\sigma)=\lambda_1\sigma,\qquad h(\sigma)=\lambda_2\sigma,\qquad  
g(\sigma)=\mu(\sigma-\tilde\sigma)
\end{equation}
with positive constants $\mu,\tilde\sigma,\lambda_1,\lambda_2$ fulfilling $\lambda_1\ge\lambda_2$. 
%
  % $\sigma=\sigma(r, t)$ and $R(t)$ are unknown functions representing 
  % the nutrient concentration and the tumor radius,  respectively.
  %  $\nu$, $\bar\sigma$ and $\sigma_Q$ are 
  % positive constants, among which  $\nu$ and $\bar\sigma$ 
  % represent the removal rate of quiescent cells and the 
  % nutrient concentration supplied by external 
  % tissues, respectively. $\sigma_Q$ represents a critical value of nutrient  
  % concentration for tumor cells transformation between the proliferating 
  % phase and the quiescent phase,  the region $\{\sigma(r, t)\le\sigma_Q\}$ 
  % consists of all quiescent cells and may form a quiescent core.
  % $f$ and $h$ are given functions representing the
  % net nutrient consumption of proliferating cells and quiescent cells within 
  % tumor, respectively, and $g$ is a given proliferating rate functions 
  % for tumor cells. $H$ is the Heaviside function, Typically, $f$, $h$ and $g$ are given by linear functions,
  % cf. [\ref{byr-cha-96}, \ref{cui-fri-01}, \ref{fri-rei-99}].
  In this paper we intend to investigate the essential characteristics of the growth of the two-layer tumors upon more general biological parameter rate functions, so we shall always assume the following conditions (A1)--(A3) hold:
\vspace{0.3em}

 \noindent $(\rm A1)$  
 $f$ and $h$ are both strictly increasing $C^1$ functions on $[\sigma_0,+\infty)$ with bounded derivatives, which also satisfy $f(\sigma_0)=h(\sigma_0)=0$ 
 and $f(\sigma_Q)\ge h(\sigma_Q)$.
% on $[\sigma_0,+\infty]$.
 % $f, h \in C^1[\sigma_0,\infty)$,  $f(\sigma_0)=h(\sigma_0)=0$, 
 % % $f'(\sigma)>0$ and $h'(\sigma)\ge 0$  
 % $f$, $h$ are both strictly increasing with bounded derivatives,  
 % and $f(\sigma)\ge h(\sigma)$ 
 % on $[\sigma_0,+\infty]$.
 \vspace{0.1em}
  
 \noindent $(\rm A2)$  
 $g$ is a strictly increasing $C^1$ function on $[\sigma_0,+\infty)$ with a unique $\tilde \sigma$ such that $g(\tilde\sigma)=0$.
 
% $g\in C^1[\sigma_0,\infty)$,  $g(\sigma)$ is strictly increasing on 
% $[\sigma_0,\infty)$, and there exists a unique $\tilde\sigma>0$ 
% such that  $g(\tilde\sigma)=0$.
  \vspace{0.1em}

\noindent $(\rm A3)$ $0\le \sigma_0<\sigma_Q<\tilde\sigma$ and
   $g(\sigma_Q)+\nu\ge 0$.
\vspace{0.3em}

The above constant $\sigma_0$ refers to a specific nutrient level at which the transport rate of nutrients via the capillary network offsets the consumption rate of nutrients by the tumor cells; cf. \cite{fri-rei-99} with $\sigma_0=\frac{\sigma_B\varGamma}{\varGamma+\lambda_0}$. 
%  
% the nutrient consumption rate by tumor cells are exactly equal to the transference rate from surrounding tissues. 
%
In (A1), the monotonicity of $f$ and $h$ signifies that the net nutrient consumption rates by the proliferating and quiescent cells are both strictly increasing for $\sigma\ge\sigma_0$.
The inequality  $f(\sigma_Q)\ge h(\sigma_Q)$ comes from the biological observation that the proliferating cells always consume nutrients more 
than the quiescent cells. 
The positive constant $\tilde\sigma$ in $(\rm A2)$ 
represents a threshold value
of nutrient concentration at which
the birth and the death rates of proliferating cells are in balance. 
%{\color{blue} The last assumption 
% reveals that the nutrient level $\sigma_Q$ is less than
% $\tilde\sigma$.}  
The last inequality in $(\rm A3)$ means the proliferating cells always
grow faster than the quiescent cells. 
All these assumptions are biologically meaningful,
cf. \cite{cui-05,liu-zhuang,wu-wang-19}. The discontinuity of both the consumption rate and proliferation rate functions makes this kind of two-phase free boundary problem very difficult.

Multicellular tumor spheroids with concentric shells of proliferating, quiescent and necrotic cells can be observed in in vitro experiments, cf. \cite{byr-cha-96}. The above model \eqref{1.1} is a simplified version of the three-layer tumor model proposed by Byrne and Chaplain \cite{byr-cha-97}. In the special case $h(\sigma)\equiv 0$, this model can be 
regarded as a necrotic tumor model, for which the global well-posedness and long-time behavior of the radial transient solutions have been established, eg., \cite{cui-06,bue-erc-08,cui-fri-01,wu-wang-19,wu-xu-20,xu-zhang-zhou}, 
% [\ref{bue-erc-08}, 
% \ref{cui-06}, \ref{cui-fri-01},  \ref{wu-wang-19}, \ref{wu-xu-20},  
% \ref{xu-zhang-zhou}]),
and moreover, the symmetry-breaking solutions and asymptotic
stability of the radial stationary solutions in non-radial 
case have also been well explored, see
\cite{hao-12,lu-hao-hu,wu-18,wu-19,wu-21}. 
If further $\sigma_Q\le\sigma_0$, then problem \eqref{1.1} becomes a one-layer tumor model with only proliferating cells. For many interesting and illuminating results of such models we refer the reader to \cite{cui-05,cui-09,cui-esc-07,esc-mat-10,fri-hu-06,fri-hu-08,fri-rei-99,he-xing-hu,hua-hu-24,hua-zha-hu-19,wu-16,zhao-hu,zhuang-cui} and the references cited therein. 

Recently,  
for the linear case \eqref{1.2}, Liu and Zhuang \cite{liu-zhuang} established the existence and uniqueness of the stationary solution and investigated the asymptotic behavior of the transient solutions to problem \eqref{1.1} by explicitly  solving the expression of nutrient concentration $\sigma(r,t)$ in terms of tumor radius $R(t)$ and quiescent core radius $\rho(t)$.  
The dynamics for this linear two-layer tumor model with a time lag in cell proliferation was considered in \cite{liu-z-2}.
%
%  $f(\sigma)$, $g(\sigma)$ and $h(\sigma)$ are 
%  all linear functions in the form of
% \begin{equation}
% f(\sigma)=\lambda_1\sigma,\qquad h(\sigma)=\lambda_2\sigma, \qquad  g(\sigma)=\mu(\sigma-\tilde\sigma),
% \end{equation}
% %(1.2)
% %   where $\mu$ and $\tilde\sigma$, $\lambda_1\le\lambda_2$ 
%   are positive 
%   constants, 
  % the nutrient concentration $\sigma(r,t)$ can be solved explicitly
  % in tumor radius $R(t)$ and quiescent core's radius $\rho(t)$, 
By assuming that the quiescent cells linearly consume nutrients at the same rate 
as the proliferating cells, but with different growth rates, Zheng, Li and Zhuang 
\cite{zheng-li-zhuang} also studied a simplified three-layer tumor model with a necrotic core. 
The dependency relationship between the two unknown interfaces was carefully analyzed and the figures on the global evolution of tumor and the mutual transitions of its cell-layer structures were shown. Still, their work was based on the explicit expression of $\sigma(r,t)$. 

In this work we give rigorous analysis to problem \eqref{1.1} with nonlinear increasing functions $f$, $h$ and $g$ with conditions $(\rm A1)$--$(\rm A3)$.
In this general case, there is no explicit expression for $\sigma$, and more relevantly, $\sigma_r$ at the inner interface turns out to be a function of the quiescent core radius $\rho$, which greatly differs from that in the necrotic tumor model where $\sigma_r$ at $\rho$ is always zero. This phenomenon makes the problem more complicated and it requires a new analysis strategy.  
We develop a general method to handle the two-layer problem of this kind, not only just available for the tumor model with a quiescent core, 
but also for the necrotic model.  
Motivated by \cite{bue-erc-08,wu-wang-19},
for any given $\rho>0$, we shall first solve the elliptic problem for $\sigma$ in quiescent region $\{r\le\rho\}$, then with $\sigma_r(\rho)$ obtained, we carry on to solve $\sigma$ in
the region $\{r>\rho\}$. By utilizing a local existence method and a shooting method, we finally obtain the tumor radius $R=R(\rho)$ in terms of $\rho$. 
Based on the maximum principle, we establish a  
comprehensive and deep analysis on the relationship between 
$\sigma$, $\rho$ and $R$, and transform problem \eqref{1.1} into an abstract Cauchy problem of ordinary differential equation. Existence and uniqueness of the stationary solution is explored and the long-time behavior of the transient solution is studied. 
We find out a critical nutrient value $\sigma^{*}\in(\tilde\sigma,+\infty)$ such that problem $(\ref{1.1})$ has a unique stationary solution with a quiescent core if and only if $\bar\sigma>\sigma^{*}$. This stationary solution is globally asymptotically stable under radial perturbations. 
% Further on, we obtain the global asymptotical stability 
% of the stationary solutions. 
Our analysis also implies that 
we can control tumor growth by adjusting the external nutrient supply $\bar\sigma$.

% \vspace{1em}

We notice that the nutrient concentration $\sigma$ is strictly convex with respect to $r$.
By fully exploiting this important geometric property, we derive some estimates on
the steady tumor radius $R_s$ and the steady quiescent core radius $\rho_s$. 
We also observe that the tumor model \eqref{1.1} with a quiescent core formally converges to the necrotic tumor model if $h(\sigma_Q)$ vanishes, which arises our interest in finding some connections between these two models. 

%
% Since the tumor model (1.1) with a quiescent core formally converges to 
% the necrotic tumor one as $h(\sigma_Q)$ converges to $0$, we are interested
% in finding   

By introducing a new auxiliary problem (\ref{4.3}), we are able to control the influence of consumption rate $h$ on the nutrient concentration $\sigma$ in the proliferating region, and prove that the stationary solution of problem \eqref{1.1} actually converges to the stationary one of the necrotic model as $h(\sigma)$ tends to zero. 
The connection between the transient solution of \eqref{1.1} and the transient one of the necrotic model possesses a similar convergent property. 
These facts suggest that in the radially symmetric setting, we can unify these two kinds of tumor models as one kind of two-layer tumor model from a mathematical view. 
      
The rest of this paper is arranged as follows.
In the next section we study stationary solutions of problem \eqref{1.1} and 
give some estimates for the stationary tumor radius $R_s$ 
under additional conditions on the proliferation rate function $g(\sigma)$.  
In Section \ref{3}, we reduce free boundary problem (\ref{1.1}) 
into a Cauchy problem of ordinary differential equation and 
study asymptotic behavior of transient solutions. In the last section we 
analyze the connection between the tumor models with a quiescent core
and with a necrotic core.

\medskip
\hskip 1em

\section{Stationary solutions}\label{2}
\setcounter{equation}{0}
\hskip 1em

In this section,  we study the existence of
stationary solutions of problem (\ref{1.1}) and derive some estimates
on the steady radius of the tumor. % stationary tumor radius.

% Similarly as the necrotic tumor model, 

As is quite similar to the case of necrotic tumor model, 
there are two kinds of stationary 
solutions to problem
(\ref{1.1}). One kind of the stationary solution has a quiescent core
with radius $\rho$,   and satisfies the following problem 

\begin{equation}\label{2.1}
\left\{
\begin{array}{l}
  \displaystyle\sigma''(r)+{\frac2r}\sigma'(r) = h(\sigma)
  \qquad\mbox{for}\;\; 0<r<\rho,
\\ [0.3 cm]
\sigma'(0)=0,\;\;\sigma(\rho)=\sigma_Q
\\ [0.3 cm]
 \displaystyle\sigma''(r)+{\frac2r}\sigma'(r) = f(\sigma)
  \qquad\mbox{for}\;\; \rho<r<R,
\\ [0.3 cm]
  \sigma'(\rho-0)=  \sigma'(\rho+0),
  \qquad \sigma(R)=\bar\sigma,
\\[0.3 cm]
  \displaystyle
  \int_\rho^R g(\sigma(r))r^2dr-{\nu\over 3}\rho^3=0.
\end{array}
\right.
\end{equation}
%(2.1)
The other kind of the stationary solution has no quiescent cores, 
and satisfies 
\begin{equation}\label{2.2}
\left\{
\begin{array}{l}
  \displaystyle\sigma''(r)+{2\over r}\sigma'(r) = f(\sigma)
  \qquad\mbox{for}\;\; 0<r<R,
\\[0.3 cm]
  \sigma'(0)=0, \quad
  \sigma(R)=\bar\sigma,\quad
\\ [0.3 cm]
\sigma(0)\ge\sigma_Q,
\\ [0.3 cm]
  \displaystyle
  \int_0^R g(\sigma(r))r^2dr=0.
\end{array}
\right.
\end{equation}
%(2.1)
In order to solve these two problems, we first consider the following
 semi-linear elliptic boundary value problem
\begin{equation}\label{2.3}
\left\{
\begin{array}{l}
  \displaystyle v''(r)+{2\over r}v'(r) = h(v(r))
  \qquad\mbox{for}\;\; 0<r<\rho,
\\ [0.3 cm]
v'(0)=0,\;\;v(\rho)=\sigma_Q.
\end{array}
\right.
\end{equation}
%(2.3)
 
\medskip

\begin{lem}\label{lem2.1}
        Suppose $(\rm A1)$ holds.
       For any $\rho>0$, problem $(\ref{2.3})$
       has a unique solution $v=V(r,\rho)$ for $ 0\le r\le \rho$
       and the following properties hold:
    
    $(i)$ $V$ is strictly increasing in $r$, and
    $$
      \sigma_0<V(r,\rho)\le\sigma_Q \qquad\mbox{for}\;\; 0\le r\le \rho.
    $$
    
    $(ii)$ $V$ is strictly convex in $r$, i.e.,
    $$
      V_{rr}(r,\rho) > 0 
      \qquad\mbox{for}\;\;  0\le r\le \rho.
    $$   
    
    $(iii)$ $V$ is strictly decreasing in $\rho$. Moreover, 
     by letting 
    $$
     \Phi(\rho):=V_r(\rho,\rho)=\displaystyle{1\over \rho^2}
     \int_0^\rho s^2h(V(s,\rho))ds\qquad \mbox{for}\;\; \rho>0,
    $$
      we have 
      \begin{equation}\label{2.3-}
      \displaystyle 0<\Phi(\rho)<{1\over 3}h(\sigma_Q)\rho\quad
      \textit{and} \quad\Phi'(\rho)>0.   
      \end{equation}

\end{lem}
%   {\bf Lemma 2.1} \ \  {\em 
% }

% \medskip

% {\bf Proof}. \  
\begin{proof} 
Since $h(v)$ is strictly increasing, we can easily verify that 
  $\sigma_0$ and $\sigma_Q$ are a pair of lower and upper solutions
  of problem $(\ref{2.3})$.
    % respectively. 
  Hence by the upper and lower solution method,
%  and the comparison principle,  
  we see that there exists
  a unique solution $v(r)=V(r,\rho)$ such that
  $\sigma_0<V(r,\rho)< \sigma_Q$ for all $0<r<\rho$.
  Due to $(\rm A1)$, we have $h(V(r,\rho))>0$, which implies
\begin{equation}\label{daoshu}
 v'(r)={1\over r^2}
 \int_0^r s^2h(V(s,\rho))ds>0
 \qquad\mbox{for}\;\;0<r\le \rho.
\end{equation}
% The assertion $(i)$ is proved.
This proves the assertion $(i)$. 
From the above relation we can also get
$$
  v''(r) -\displaystyle\frac{v'(r)}{r}=h(V(r,\rho))
  -\frac{3}{r^3}\int_0^r s^2h(V(s,\rho))ds\ge 0,
$$
and the assertion $(ii)$ follows.

  Denote $\displaystyle v_\rho(r)={\partial V\over \partial \rho}(r,\rho)$. From (\ref{2.3}) we see $v_\rho$ satisfies
\begin{equation*}
\left\{
\begin{array}{l}
  \displaystyle(v_{\rho})_{rr}+{2\over r}(v_{\rho})_r = h'(V)v_{\rho}
  \qquad\mbox{for}\;\; 0<r<\rho,
\\ [0.3 cm]
(v_{\rho})_r(0)=0,\;\; v_{\rho}(\rho)=-V_r(\rho,\rho)<0.
\end{array}
\right.
\end{equation*}
  Then by the maximum principle we have $v_\rho(r)<0$ for $0<r<\rho$.
  Similarly, there holds
\begin{equation*}
\left\{
\begin{array}{l}
  \displaystyle(v_r)_{rr}+{2\over r}(v_r)_r-\frac{2}{r^2}v_r
  = h'(V) v_r
  \qquad\mbox{for}\;\; 0<r<\rho,
\\ [0.3 cm]
v_r(0)=0,\;\;v_r(\rho)>0.
\end{array}
\right.
\end{equation*}
  Denote $w(r)=v_r(r)+v_{\rho}(r)$.
  It satisfies
\begin{equation*}
\left\{
\begin{array}{l}
  \displaystyle w_{rr}+{2\over r}w_r= h'(V)w+\frac{2}{r^2} v_r
  \qquad\mbox{for}\;\; 0<r<\rho,
\\ [0.3 cm]
w(\rho)=0.
\end{array}
\right.
\end{equation*}
Note that $h'(V)\ge0$ and $v_r>0$, then by using the strong
maximum principle, we obtain
$\Phi'(\rho)=w_r(\rho)>0$.
 % The proof is complete. 
The relation $0<\Phi(\rho)<{\frac{1}{3}}h(\sigma_Q)\rho$ comes from \eqref{daoshu}.
This completes the proof. 
 \end{proof}
% By the above equation,  we can also get
% $$
%   \sigma_{rr} -\displaystyle\frac{\sigma_{r}}{r}=h(V(r,\rho))
%   -\frac{3}{r^3}\int_0^r s^2h(V(s,\rho))ds\ge 0,
% $$
% and the assertion $(ii)$ follows.

%   Denote $\displaystyle \sigma_\rho={\partial V\over \partial \rho}(r,\rho)$. From (\ref{2.3}), we see $\sigma_\rho$ satisfies
% \begin{equation*}
% \left\{
% \begin{array}{l}
%   \displaystyle(\sigma_{\rho})_{rr}+{2\over r}(\sigma_{\rho})_r = h'(V)\sigma_{\rho}
%   \qquad\mbox{for}\;\; 0<r<\rho,
% \\ [0.3 cm]
% (\sigma_{\rho})_r(0)=0,\;\;\sigma_{\rho}(\rho)=-V_r(\rho,\rho)<0.
% \end{array}
% \right.
% \end{equation*}
%   Then by the maximum principle we have $\sigma_\rho(r,\rho)<0$ for $0<r<\rho$.
%   Similarly, there holds
% \begin{equation*}
% \left\{
% \begin{array}{l}
%   \displaystyle(\sigma_r)_{rr}+{2\over r}(\sigma_r)_r-\frac{2}{r^2}\sigma_r 
%   = h'(V)\sigma_r
%   \qquad\mbox{for}\;\; 0<r<\rho,
% \\ [0.3 cm]
% \sigma_r(0)=0,\;\;\sigma_r(\rho)>0.
% \end{array}
% \right.
% \end{equation*}
%   Denote $w(r)=\sigma_r(r)+\sigma_{\rho}(r)$.  
%   It satisfies
% \begin{equation*}
% \left\{
% \begin{array}{l}
%   \displaystyle w_{rr}+{2\over r}w_r= h'(\sigma)w+\frac{2}{r^2}\sigma_r
%   \qquad\mbox{for}\;\; 0<r<\rho,
% \\ [0.3 cm]
% w(\rho)=0.
% \end{array}
% \right.
% \end{equation*}
% Note that $h'(\sigma)\ge0$ and $\sigma_r\ge 0$, then by using the strong maximum principle, we obtain
% $\Phi'(\rho)=w_r(\rho)>0$.
%  The proof is complete. % \hfill$\Box$
% \end{proof}
% \bigskip

 Since $\displaystyle \Phi(\rho)={1\over \rho^2}\int_0^\rho s^2h(V(s,\rho))ds$ for 
 $\rho>0$,  by continuity we let $\Phi(0)=0$.

Next, we employ a shooting method to solve $\sigma$ in the proliferating region
$\{\rho<r<R\}$.  For any given $\rho\ge0$, 
%we further consider the following problem:
%\begin{equation}\label{2.4}
%\left\{
%\begin{array}{l}
% \displaystyle\sigma''(r)+{2\over r}\sigma'(r) = f(\sigma)
%  \qquad\mbox{for}\;\; r>\rho,
%\\ [0.3 cm]
%  \sigma(\rho)=\sigma_Q,
%\\ [0.3 cm]
%  \sigma'(\rho)=  \Phi(\rho),
%\\[0.3 cm]
%  \sigma(R)=\bar\sigma.
%\end{array}
%\right.
%\end{equation}
%(2.4)
 we consider the following Cauchy problem 
\begin{equation}\label{2.5}
\left\{
\begin{array}{l}
 \displaystyle v''(r)+{2\over r}v'(r) = f(v(r))
  \qquad\mbox{for}\;\; r>\rho,
\\ [0.3 cm]
  v(\rho)=\sigma_Q,
\\ [0.3 cm]
  v'(\rho)=  \Phi(\rho).
\end{array}
\right.
\end{equation}
%(2.5)

% We have 

\medskip

% {\bf Lemma 2.2} \ \ {\em 
\begin{lem}\label{lem2.2}
Suppose $(\rm A1)$ holds.
   For any $\rho\ge0$,  problem $(\ref{2.5})$
   has a unique maximal solution $v=\mathcal V(r,\rho)\in
   C^2(\left[\rho, R_{\infty}\right); \mathbb{R})$ and the
   following properties hold:

 $(i)$ $\mathcal V$ is strictly increasing in $r$,
  and
  $$\displaystyle \lim_{r\to R_{\infty}-0}\mathcal V(r,\rho)=+\infty.$$

   $(ii)$ $\mathcal V$ is strictly convex in $r$, i.e.,
$$
  \mathcal V_{rr}(r,\rho) > 0 \qquad\mbox{for}\;\;
  \rho\le  r<R_{\infty}.
$$

$(iii)$ $\mathcal V$ and $\mathcal V_r$ are both strictly decreasing
in $\rho$.
\end{lem}

% \medskip

% {\bf Proof}. 
\begin{proof}

We first show the local existence and uniqueness of problem
(\ref{2.5}).  Let $b>0$ and
$$
M:=\max_{[\sigma_Q,\sigma_Q+b]}\{|f(u)|+|f'(u)|\},
%+\max_{[0,\sigma_Q]} |h'(u)|,
\qquad
h:=\min\Big\{\frac{b}{2h(\sigma_Q)(\rho+1)},\sqrt{\frac{1}{3M}},\sqrt{\frac{b}{2M}}
\Big\}.
$$
  Set
$$
 \mathscr{B}:=\{u\in C([\rho,\rho+h]): \sigma_Q\le u(r)\le\sigma_Q+b\}.
$$
 Obviously,  $\mathscr{B}$ is a closed subspace of Banach space $C({[\rho,\rho+h]})$.
One can easily verify that problem $(\ref{2.5})$ is equivalent to the following integral equation
\begin{equation}\label{2.6}
v(r)=\sigma_Q+ \rho^2 \Phi(\rho)(\frac{1}{\rho}-\frac{1}{r})+\int_{\rho}^r\frac{1}{\alpha^2}\int_{\rho}^\alpha s^2f(v(s))dsd\alpha.
\end{equation}
Define a mapping $T: \mathscr{B}
\rightarrow C([\rho,\rho+h])$ by
$$
 Tu(r):=\sigma_Q+ \rho^2\Phi(\rho)(\frac{1}{\rho}-\frac{1}{r})+\int_{\rho}^r\frac{1}{\alpha^2}\int_{\rho}^\alpha s^2f(u(s))dsd\alpha \qquad
 \mbox{for} \;\; u\in \mathscr{B}.
$$
Then for any $u\in\mathscr{B}$,  it follows from (A1) and
$\eqref{2.3-}$ that
\begin{equation*}
  \sigma_Q\le   Tu(r) \le \sigma_Q+\frac{1}{3}h(\sigma_Q)\rho h+Mh^2<\sigma_Q+b     \;\;\mbox{for}\;\; \rho\le r\le \rho+h,
\end{equation*}
 hence $T(\mathscr{B})\subset \mathscr{B}$.
 Moreover, $T$ is a contraction mapping for the reason that
\begin{equation*}
\begin{matrix}
\begin{aligned}
  \Vert Tu-Tv \Vert \ &\le \sup_{[\rho,\rho+h]}
  \int_{\rho}^r\frac{1}{\alpha^2}\int_{\rho}^\alpha s^2\vert f(u(s))-f(v(s))\vert
   dsd\alpha   \\
    &\le Mh^2\Vert u-v\Vert \le\frac{1}{3}\Vert u-v\Vert  \qquad
 \mbox{for} \;\; u, v\in \mathscr{B}.
\end{aligned}
\end{matrix}
\end{equation*}
By Banach fixed point theorem,
 $T$ has a unique fixed point  $v(r)=\mathcal V(r,\rho)$ in $\mathscr{B}$,
 which is  the unique local solution of problem $(\ref{2.5})$.

 Proceeding as in Lemma 2.1 $(i)$, we see that $\mathcal V$ and $\mathcal V_r$
 are both strictly increasing in $r$ for $r\ge\rho$. By extension we obtain the maximal existence interval $[\rho, R_{\infty})$, where $R_{\infty}\in(\rho,+\infty]$,
 and we easily get the assertion $(i)$.

On the other hand, from $(\ref{2.6})$ and (A1), we also have
\begin{equation*}
\begin{matrix}
\begin{aligned}
\mathcal V_{rr}(r,\rho) -\displaystyle\frac{\mathcal V_{r}(r,\rho)}{r}&
=f(\mathcal V(r,\rho))-\displaystyle\frac{3}{r^3}\rho^2\Phi(\rho)-\displaystyle\frac{3}{r^3}\int_{\rho}^{r}s^2
f(\mathcal V(s,\rho))ds \\
&\ge\displaystyle\frac{\rho^3}{r^3}(f(\sigma_Q)-h(\sigma_Q))\ge 0.
\end{aligned}
\end{matrix}
\end{equation*}
Hence the assertion $(ii)$ holds.

Finally, we observe that for $r>\rho$,
\begin{equation}\label{Vr}
    r^2\mathcal V_r(r,\rho)=\rho^2\Phi(\rho)+\int_\rho^r s^2f(\mathcal V(s,\rho)) ds
    =\int_0^\rho s^2 h(V(s,\rho)) ds +\int_\rho^r s^2f(\mathcal V(s,\rho)) ds.
\end{equation}
Then by (A1), Lemma 2.1 $(iii)$ and a comparison argument,
 we easily get
 that $\mathcal V_r(r,\rho)$ is strictly decreasing in $\rho$,
 and thus $\mathcal V(r,\rho)$ is also strictly decreasing.
 In fact,
 if $\rho_1<\rho_2$, from \eqref{Vr} we have for $r>\rho_2$,
\begin{align*}
     r^2\mathcal{V}_r(r,\rho_1)-r^2\mathcal{V}_r(r,\rho_2)= & \,\, \int_0^{\rho_1} s^2h(V(s,\rho_1)) ds + \left(\int_{\rho_1}^{\rho_2}+\int_{\rho_2}^r\right) s^2 f(\mathcal{V}
    (s,\rho_1)) ds \\
     & \,- \left(\int_0^{\rho_1} +\int_{\rho_1}^{\rho_2}\right) s^2h(V(s,\rho_2))ds - \int_{\rho_2}^rs^2 f(\mathcal{V}(s,\rho_2)) ds.
\end{align*}
Since $V(s,\rho_1)>V(s,\rho_2)$ ensured by Lemma \ref{2.1} for $s\in(0,\rho_1)$ and $\mathcal{V}(s,\rho_1)>\sigma_Q>V(s,\rho_2)$ for $s\in(\rho_1,\rho_2)$, we deduce from the above relation and the properties of $f$ and $h$ that 
\begin{equation*}
\left\{
\begin{array}{l}
\,\mathcal{V}_r(r,\rho_1)-\mathcal{V}_r(r,\rho_2)> \displaystyle\frac1{r^2}\int_{\rho_2}^r s^2 \big[f(\mathcal{V}(s,\rho_1))-f(\mathcal{V}(s,\rho_2))\big] ds\quad \mbox{for}\;\;  r>\rho_2,
  \\ [0.3 cm]
\,\mathcal{V}(\rho_2,\rho_1)-\mathcal{V}(\rho_2,\rho_2)>0. 
\end{array}
\right.
\end{equation*}
This implies that both $\mathcal{V}(r,\rho_1)>\mathcal{V}(r,\rho_2)$ and $\mathcal{V}_r(r,\rho_1)>\mathcal{V}_r(r,\rho_2)$ for all $r>\rho_2$ by iteration argument. 
It follows that $\mathcal{V}(r,\rho)$ and $\mathcal{V}_r(r,\rho)$ are both decreasing in $\rho$.  The assertion $(iii)$ follows and
the proof is complete.\end{proof}

 \bigskip

From the above results, we see that for any given $\rho\ge 0$ and
  $\bar\sigma>\sigma_Q$,  there exists a unique $R=R(\rho)>\rho$ such that
\begin{equation}\label{shoot}
\mathcal V(R(\rho),\rho)=\bar\sigma.
\end{equation}
Then  problem $(\ref{2.1})_1$--$(\ref{2.1})_4$ admits a unique solution
$(\sigma, R)=(\Sigma(r,\rho,R(\rho)), R(\rho))$ with $R(\rho)\ge \rho$,
where
\begin{equation}
\Sigma(r,\rho, R)=\left\{
\begin{array}{rl}
  V(r,\rho)\quad \mbox{for}\;\; 0\le r\le \rho,
  \\ [0.3 cm]
  \mathcal V(r,\rho)\quad \mbox{for}\;\; \rho<r\le R.
\end{array}
\right.
\end{equation}
 From Lemma 2.2 $(iii)$,  we have  $R(\rho)$ is strictly increasing in $\rho$.
Define
\begin{equation}\label{Rc}
R_c:=\lim_{\rho\to 0^+} R(\rho).
\end{equation}
%(Rc)
 It is the critical radius such that problem $(\ref{2.2})_1$--$(\ref{2.2})_2$
 with $\rho=0$
 has a unique solution satisfying $\sigma(0)=\sigma_Q$ for $R=R_c$.  Obviously,
 $R_c>0$.

 Clearly,  the mapping $\rho\mapsto R(\rho)$ is a 1-1 correspondence.  So we can also regard $\rho$ as a function of the variable $R$, i.e.,
 $\rho=\rho(R)$ for
 $R\ge R_c$,  and rewrite the solution
 $\sigma=\Sigma(r,\rho(R),R)$ for $\rho(R)\le r\le R$.

We make the variable transformation $s=r/R$ and denote
$$
\eta(R)=\rho(R)/R, \qquad U(s,\eta(R),R)=\Sigma(sR,\eta(R)R,R).
$$
 Then $u(s)=U(s,\eta(R),R)$ satisfies the following problem
\begin{equation}\label{2.7}
\left\{
\begin{array}{l}
 \displaystyle u''(s)+{2\over s}u'(s) = R^2f(u)
  \qquad\mbox{for}\;\; \eta<s<1,
\\ [0.3 cm]
  u(\eta)=\sigma_Q,
\\ [0.3 cm]
  u'(\eta)=  R\Phi(\eta R),
\\[0.3 cm]
  u(1)=\bar\sigma.
\end{array}
\right.
\end{equation}
\medskip

  % {\bf Lemma 2.3} \ \  {\em 
\begin{lem}\label{lem2.3}
  Suppose $(\rm A1)$ and
   $\bar\sigma>\sigma_Q$ hold. For any $R\ge R_c$,  problem $(\ref{2.7})$
   has a unique solution $(u,\eta)=(U(s,\eta(R), R), \eta(R))$, 
   and the following
   properties hold:
   
   $(i)$ $U(s,\eta(R), R)$ is strictly increasing in $s$,
   and strictly decreasing in R.
 %  , i.e.
%$${\partial U\over\partial R}(s,\eta(R),R)<0.$$ 

  $(ii)$ $U(s,\eta(R), R)$ is strictly convex in $s$ and satisfies
$$
U_{ss}(s,\eta(R),R)\ge\frac{1}{s}U_{s}(s,\eta(R),R)>0.
$$

  $(iii)$ $\eta(R)$ is continuous and strictly increasing for $R\ge R_c$,
  $\eta(R_c)=0$ and  $\displaystyle\lim_{R\to+\infty}\eta(R)=1$.
\end{lem}

% \medskip

% {\bf Proof}. \  
\begin{proof}
The proofs of the monotonicity and the convexity of 
$U(s,\eta(R), R)$ in $s$ are 
similar as in Lemma \ref{lem2.1}, we omit it here. In the following, we mainly show the monotonicity 
of $U(s,\eta(R), R)$ in $R$.

Denote $u(s,R)=U(s,\eta(R),R)$ and take
$$
  z(s)={\partial u\over\partial R}(s,R),\qquad \xi=\eta'(R).
$$
By a direct computation,   
  $z(s)$ and $\xi$ satisfy the following problem
\begin{equation}\label{2.8}
\left\{
\begin{array}{l}
  \displaystyle z''(s)+{2\over s}z'(s)=R^2f'(u)z + 2 R f(u)
  \qquad\mbox{for}\;\; \eta<s<1,
\\ [0.3 cm]  
  z(1)=0,\qquad z(\eta)=-R\Phi(\eta R)\xi,
\\ [0.3 cm]
  z'(\eta)=\Phi(\eta R)+\Phi'(\eta R)\eta R+(h(\sigma_Q)-f(\sigma_Q)
  +\Psi(\eta R))R^2\xi,
\end{array}
\right.
\end{equation}
%(2.11)
   where $\eta=\eta(R)$ and 
$$
  \Psi(\rho)={1\over \rho^2}\int_0^\rho s^2 h'(V(s,\rho))V_\rho(s,\rho) ds.
$$ 
  Here we have used the relation
$$
  \Phi'(\rho)=-{2\over \rho}\Phi(\rho)+h(\sigma_Q)+\Psi(\rho).
$$
  In fact, from $(\ref{2.7})_3$ we have 
$$
  z'(\eta)=\Phi(\eta R)+R\Phi'(\eta R)(\eta' R+\eta)-u''(\eta,R)\eta'.
$$
  This relation together with $(\ref{2.7})_1$ leads to the last equation
  of $(\ref{2.8})$.
   If $\xi\le 0$,  by $(\rm A1)$ and Lemma \ref{lem2.1} $(iii)$ 
   we see $z(\eta)\ge0$, and by $\Psi(\rho)\le 0$ we also have
  $z'(\eta)>0$. On the other hand, since $f(u)>0$,  by the strong maximum
   principle, we have $z'(\eta)<0$. Hence
  $\xi>0$ and then $z(\eta)<0$. Again by the maximum principle, we get   
\begin{equation}\label{2.9}
  z(s)={\partial u\over\partial R}(s,R)<0 \quad\mbox{for}\;\; \eta<s<1,  
  \qquad\mbox{and}\qquad
  \xi=\eta'(R)>0.
\end{equation}
%It implies that $Z(s,R)$ is strictly decreasing in $R$, and we also see that
%$\eta(R)$ is strictly increasing in $R$ for $R\ge R_c$. 
Then the monotonicity of $u(s,R)$ follows.
Finally, from $(\ref{2.7})$ we have
%$$
%u(s)=\sigma_Q+R^2\int_0^{\eta(R)}l^2h(\sigma(lR))dl(\frac{1}{\eta(R)}-\frac{1}{s})
%+R^2\int_{\eta(R)}^s\frac{1}{\alpha^2}\int_{\eta(R)}^{\alpha}l^2f(u(l))dld\alpha.
%$$ 
%Combining the last line of $(\ref{2.7})$, there follows the subscript equation:
$$
  \frac{\bar{\sigma}-\sigma_Q}{R^2}={\eta^2(R)\over R}\Big(\frac{1}  
  {\eta(R)}-1\Big)\Phi(\eta(R)R)
  +\int_{\eta(R)}^1\frac{1}{\alpha^2}\int_{\eta(R)}^{\alpha}s^2f(u(s,R))dsd\alpha.
$$
By letting $R\to+\infty$ and observing $u\ge\sigma_Q>\sigma_0$, we deduce $\displaystyle\lim_{R\to+\infty}\eta(R)=1$.
This completes the proof. 
  % The proof is complete.       
\end{proof} %\hfill$\Box$
 
 % \bigskip

  From Lemma \ref{lem2.2} and Lemma \ref{lem2.3}, we conclude that
  if and only if  $R\ge R_c$,
  problem $(\ref{2.1})_1$-$(\ref{2.1})_4$
  has a unique solution
  $(\sigma,\rho)=(\Sigma(r, \rho(R), R), \rho(R))$
  where $\rho(R)= \eta(R)R$.

Define
\begin{equation}\label{2.10}
  F(R):={1\over R^3}\Big[\int_{\rho(R)}^{R}g(\Sigma(r,\rho(R), R))r^2 dr-
  {\nu \over 3}\rho^3(R)\Big]
  \qquad \mbox{for} \;\;R\ge R_c.
\end{equation}
Then problem $(\ref{2.1})$ is equivalent to equation $F(R)=0$.

\medskip

  % {\bf Lemma 2.4} \ \  {\em 
\begin{lem}\label{lem2.4}
  Suppose $(\rm A1)$--$(\rm A3)$ hold. 
  Then there exists a critical value $\sigma^*\in (\tilde\sigma, +\infty)$ such that
  
  $(i)$ For $\bar\sigma\ge\sigma^*$, equation $F(R)=0$ has a unique  
  root $R_s\in [R_c,+\infty)$.
  
  $(ii)$ For 
  $\bar\sigma<\sigma^*$, equation $F(R)=0$ has 
  no roots in $[R_c, +\infty)$.  
\end{lem}

% \medskip

 % {\bf Proof}. \  
\begin{proof}
 We rewrite 

$$
  F(R)
  =  \int_{\eta(R)}^1 g(u(s,R))s^2ds-{\nu\over3}\eta^3(R)
  \qquad \mbox{for}\;\; R\ge R_c.
$$
  From Lemma \ref{lem2.3} $(iii)$, we have 
\begin{equation} \label{2.12}
\lim_{R\to+\infty} F(R)=-{\nu\over3}<0.
\end{equation}
    By $(\rm A2)$, $(\rm A3)$ and (\ref{2.9}), we have  for any $R>R_c$,
 \begin{equation}\label{2.13}
   F'(R)=\int_{\eta(R)}^1 g'(u(s,R)){\partial u\over \partial R}(s)s^2ds
   -(g(\sigma_Q)+\nu)\eta^2(R)\eta'(R)
  <0.
 \end{equation}
 %(2.13) 
 % Take $\tilde\sigma>\sigma_Q>\sigma_0$ be fixed. 
 Fix $\tilde\sigma>\sigma_Q>\sigma_0$ and take $\bar\sigma$ as a parameter. 
 We recall the definition 
 of $R_c$ in (\ref{Rc}) and regard  $R_c=R_c(\bar\sigma)$ as a function of $\bar\sigma$ 
  for $\bar\sigma\in(\sigma_Q,+\infty)$. 

Denote
$$
  Y(s,\bar\sigma):=U(s,0,R_c(\bar\sigma)) \qquad\mbox{and}\qquad
  \mathcal F(\bar\sigma):=F(R_c(\bar\sigma))=\int_0^1 g(Y(s,\bar\sigma))s^2ds.
$$
  Obviously, we have $\mathcal F(\tilde\sigma)<g(\tilde \sigma)=0$.
Let $\displaystyle y(s)={\partial Y\over\partial \bar\sigma}(s,\bar\sigma)$ and 
  $\psi=R_c'(\bar\sigma)$,
  one can verify that 
\begin{equation}
\left\{ 
\begin{array}{l}
  \displaystyle y''(s)+{2\over s}y'(s)=R_c(\bar\sigma)^2f'(Y)y+2 R_c(\bar\sigma) \psi f(Y)
  \qquad\mbox{for}\;\; 0<s<1,
\\ [0.3 cm]  
  y(1)=1,\qquad y(0)=0,\qquad
  y'(0)=0.
\end{array}
\right.
\end{equation}
%(2.14)
Similarly to (\ref{2.9}), we obtain
\begin{equation}\label{2.14}
  y(s)={\partial Y\over\partial \bar\sigma}(s,\bar\sigma)> 0\qquad\mbox{and}\qquad
  \psi=R_c'(\bar\sigma)>0.
\end{equation}
%(2.16)  
This implies that
\begin{equation}\label{2.15}
\mathcal F'(\bar\sigma)=\int_0^1 g'(Y(s,\bar\sigma))y(s)s^2ds>0
\qquad\mbox{for}\;\;\bar\sigma>\sigma_Q.
\end{equation}
   Moreover, following a similar proof of Theorem \cite{wu-wang-19}, 
   one can easily obtain from the boundedness of $f'$ that 
   $\displaystyle \lim_{\bar\sigma\to+\infty}R_c(\bar\sigma)= +\infty$
   and $\mathcal F(+\infty)>0$. 
 Hence 
 we can find that there exists a unique  $\sigma^*>\tilde\sigma$ such that 
% $\boldsymbol{\Theta}(\sigma^*)=0$. 
\begin{equation}\label{2.16}
 F(R_c(\bar\sigma))= \mathcal F(\bar\sigma)
  \left\{
  \begin{array}{ll}
  >0,\qquad  \bar\sigma>\sigma^*,
  \\ %[0.2 cm]
   =0, \qquad \bar\sigma=\sigma^*,
   \\ %[0.2 cm]
   <0,\qquad \sigma_Q<\bar\sigma<\sigma^*.
\end{array}
\right.
\end{equation}
%(2.17)
Therefore, we conclude that, for $\bar\sigma\ge\sigma^*$, the equation $F(R)=0$ has a unique solution 
$R_s\ge R_c$, whereas for $\sigma_Q\le\bar\sigma<\sigma^*$, the equation $F(R)=0$ has no solutions in $[R_c,+\infty)$.
\end{proof}
% \hfill$\Box$

\medskip

 From Lemma \ref{lem2.1} and Lemma \ref{lem2.2},   we see that for any given 
 $R\in (0,R_c(\bar\sigma))$,   the equations $(\ref{2.2})_1$--$(\ref{2.2})_2$ admit a 
 unique solution $\sigma(r)=W(r,R)$ with $\sigma_Q<W(r,R)\le \bar\sigma$. Note that
 $W(r,R)$ is strictly increasing  and convex in $r$, and strictly decreasing
 in $R$.  Define
\begin{equation}\label{HR}
  F(R):={1\over R^3}\int_0^R g(W(r,R)) r^2 dr\qquad\mbox{for} \;\; 0<R<R_c.
\end{equation}
%(2.18)
  Clearly, combining (\ref{2.10}) and (\ref{HR}) we know that $F(R)$ is continuous and strictly decreasing for $R\in(0,+\infty)$.
 Moreover, we have for $R< R_c$,
\begin{equation}\label{2.19}
  F'(R)<0,\qquad
  \lim_{R\to0^+}F(R)={1\over3}g(\bar\sigma)
  \left\{
  \begin{array}{l}
  >0,\quad \bar\sigma>\tilde\sigma,
  \\
  \le 0,\quad \bar\sigma\le\tilde\sigma.
  \end{array}
  \right.
\end{equation}
%(2.19)
  Therefore, equation  $F(R)=0$ has a unique root $R_s\in (0,R_c)$ if and 
  only if 
  $\tilde\sigma<\bar\sigma<\sigma^*$.
\medskip

  Since the stationary problem of (1.1) is equivalent to equation $F(R)=0$, 
   we immediately obtain from Lemmas \ref{lem2.1}--\ref{lem2.4} that
% \medskip

% {\bf Theorem 2.5} \ \ {\em 

\begin{thm}\label{thm2.5}
Suppose $(\rm A1)$--$(\rm A3)$ hold and let $\sigma^*\in(\tilde\sigma,+\infty)$ be ensured in Lemma $\textsl{\ref{lem2.4}}$. Then
 % there exists a critical 
  % value $\sigma^*\in (\tilde\sigma, +\infty)$ such that  
  
  $(i)$ If $\bar\sigma>\sigma^*$, problem $(\ref{1.1})$ has a unique stationary 
  solution  $(\sigma_s(r), \rho_s, R_s)$ with a quiescent core, where 
  $\rho_s=\rho(R_s)$ is the quiescent core radius and $R_s$ 
  is the unique root of $F(R)=0$ in $(R_c,+\infty)$.

  $(ii)$ If $\tilde\sigma<\bar\sigma\le \sigma^*$, problem
  $(\ref{1.1})$ has a unique stationary solution with
  only proliferating cells $(\sigma_s(r),R_s)$, where 
  $R_s$ is the unique root of $F(R)=0$ in $(0,R_c]$.

  $(iii)$ If $\bar\sigma\le \tilde\sigma$, problem $(\ref{1.1})$ has only 
  trivial stationary solution with $R_s=0$. 
\end{thm}

In the sequel, we are interested in the stationary solution with a quiescent core. 
By using the monotonicity and the convexity of $V$, $\mathcal{V}$ and $U$ proved in
Lemmas \ref{lem2.1}--\ref{lem2.3},  we give some estimates for the steady tumor 
radius $R_s$ and for $\eta_s:=\rho_s/R_s$ with the quiescent core radius $\rho_s$.
\medskip

 % {\bf Proposition 2.6} \  {\em  

\begin{prop}\label{prop2.6}
 Assume that $\bar\sigma>\sigma^{*}$.  If there exists $\beta\in(\sigma_0,\sigma_Q)$ such that 
\begin{equation}\label{2.26} 
\int_{\beta}^{\bar\sigma}g(t)(t-\beta)^2dt=0,
\end{equation}
 and 
\begin{equation} \label{nu}
  \nu\ge\frac{3}{(\sigma_Q-\beta)^3}
  \int_{\beta}^{\sigma_Q}\vert g(t)\vert (t-\beta)^2dt,
\end{equation}
then we have
\begin{equation}\label{2.27}
\eta_s\le\eta_{\beta}:=\frac{\sigma_Q-\beta}{\bar\sigma-\beta},
\end{equation}
 where $\eta_s=\eta(R_s)$, and
\begin{equation}\label{2.28}
\frac{6(\bar\sigma-\sigma_Q)}{f(\bar\sigma)}\le R_s^2\le\frac{6(\bar\sigma-\beta)}{f(\sigma_Q)(1-\eta_{\beta}^3)}.
\end{equation}
\end{prop}

% \medskip
 
 % {\bf Proof}. \  

\begin{proof}
 Since $\eta_\beta\in (0,1)$, it follows from Lemma \ref{lem2.3} $(iii)$ that 
  there exists a unique $R_{\beta}>R_c$ such that $\eta(R_{\beta})=\eta_{\beta}$.
  Denote by 
  $u_{\beta}(s)=U(s,\eta(R_{\beta}), R_{\beta})$ $(\eta_\beta\le s
  \le 1)$ 
  the corresponding solution of problem (\ref{2.7}) with $R=R_\beta$.

  Define $\varphi(s):=(\bar\sigma-\beta)s+\beta$ for $\eta_\beta\le s
  \le 1$. 
  From $\varphi(\eta_\beta)=\sigma_Q$
  and $u''_\beta(s)>0$,
  we have $u_\beta(s)\le \varphi(s)$. Then by using $(\rm A2)$, $(\rm A3)$,  $(\ref{2.26})$ and 
  $(\ref{nu})$,
we get
\begin{equation*}
\begin{matrix}
\begin{aligned}
    F(R_{\beta})&=\int_{\eta_{\beta}}^1 g(u_\beta(s))s^2ds
    -{\nu\over3}\eta^3_{\beta}  \\
    &< \int_{\eta_{\beta}}^1 g(\varphi(s))s^2ds-\frac{1}
    {(\bar\sigma-\beta)^3}\int_{\beta}^{\sigma_Q}\vert g(t)\vert (t-\beta)^2dt 
   % (\frac{\sigma_Q-\beta}{\bar\sigma-\beta})^3
    \\
    &\le\frac{1}{(\bar\sigma-\beta)^3}\int_{\sigma_Q}^{\bar\sigma}g(t)(t-\beta)^2dt-\frac{1}{(\bar\sigma-\beta)^3}\int_{\beta}^{\sigma_Q}\vert g(t)\vert (t-\beta)^2dt\\
&\le\frac{1}{(\bar\sigma-\beta)^3}\int_{\beta}^{\bar\sigma}g(t)(t-\beta)^2dt=0,
\end{aligned} 

\end{matrix}
\end{equation*}
  It follows that $R_c<R_s<R_{\beta}$. Then the monotonicity of $\eta(R)$ ensured by Lemma \ref{lem2.3} $(iii)$ yields the estimate $(\ref{2.27})$.
 
 By recalling $\eta(R_c)=0$ we obtain
$$
\sigma_Q=U(0,0,R_c)=\bar\sigma-R^2_c\int_0^1\frac{1}{\alpha^2}\int_0^{\alpha}s^2f(U(s,0,R_c))ds d\alpha\ge\bar\sigma-\frac{1}{6}R^2_cf(\bar\sigma).
$$
  By a direct computation,  from $(\rm A1)$ and 
  $u_\beta(s)\ge\sigma_Q$ for $\eta_\beta\le s\le 1$, 
  we have 
\begin{equation*}
\begin{matrix}
\begin{aligned}
  \bar\sigma&=\sigma_Q+R^2_{\beta}(\frac{1}{\eta_{\beta}}-1) \int_0^{\eta_{\beta}}  
  s^2h(V(sR_{\beta},\eta_\beta R_\beta))ds
  +R^2_{\beta}\int_{\eta_{\beta}}^1\frac{1}{\alpha^2}\int_{\eta_{\beta}}^{\alpha}
  s^2f(u_{\beta}(s))ds d\alpha
  \\
  &>\sigma_Q+R^2_{\beta}f(\sigma_Q)
  \int_{\eta_{\beta}}^1\frac{1}  
  {\alpha^2}\int_{\eta_{\beta}}^{\alpha}s^2ds d\alpha
  \\
 &=\sigma_Q+\frac{1}{6}R^2_{\beta}f(\sigma_Q)(1-\eta_{\beta})^2(1+2\eta_{\beta})\\
&\ge\sigma_Q+\frac{1}{6}R^2_{\beta}f(\sigma_Q)(1-\eta^3_{\beta})(1-\eta_{\beta})
\\
&=\sigma_Q+\frac{1}{6}R^2_{\beta}f(\sigma_Q)(1-\eta^3_{\beta})
  {\bar\sigma-\sigma_Q\over \bar\sigma-\beta }.
\end{aligned} 
\end{matrix}
\end{equation*}
  Then estimate $(\ref{2.28})$ follows immediately from 
$R_c^2<R_s^2<R_{\beta}^2$.    
\end{proof}

% \hfill$\Box$

% \medskip

In linear function case $g(\sigma)=\mu(\sigma-\tilde\sigma)$, if $\sigma_0<4\tilde\sigma-3\bar\sigma<\sigma_Q$, we can easily verify that 
(\ref{2.26}) is satisfied by taking $\beta=4\tilde\sigma-3\bar\sigma$.
  
 For the case where (\ref{2.26}) does not hold,  we provide the following
estimates.

 \medskip

 % {\bf Proposition 2.7} \  

\begin{prop}\label{prop2.7}
 Assume that  $\bar\sigma>\sigma^{*}$. If there exists 
  $\delta\in(\sigma_0,\sigma_Q)$ such that  
\begin{equation}
\nu\ge\frac{3}{(\sigma_Q-\delta)^3}\int_{\sigma_Q}^{\bar\sigma}g(t)(t-\delta)^2dt,
\end{equation}
then we have
\begin{equation}\label{2.31}
\eta_s\le\eta_{\delta}:=\frac{\sigma_Q-\delta}{\bar\sigma-\delta},
\end{equation}

\begin{equation}\label{2.32}
\frac{6(\bar\sigma-\sigma_Q)}{f(\bar\sigma)}\le R_s^2\le
\frac{6(\bar\sigma-\delta)}{f(\sigma_Q)(1-\eta_{\delta}^2)}.
\end{equation}
\end{prop}

% \medskip

 % {\bf Proof}. \ 

\begin{proof}
Note that $\eta_\delta\in(0,1)$. Then there exists a $R_{\delta}>R_c$ 
such that $\eta(R_{\delta})=\eta_{\delta}$.  Let 
$\phi(s)=(\bar\sigma-\delta)s+\delta$ and 
$u_\delta(s)=U(s,\eta(R_{\delta}),R_{\delta})$.  
From $\phi(\eta_\delta)=\sigma_Q$, $\phi(1)=\bar\sigma$ and the 
convexity of $u_\delta(s)$,
we see that $u_\delta(s)<\phi(s)$ for $\eta_\delta<s<1$, then
\begin{equation*}
\begin{matrix}
\begin{aligned}
    F(R_{\delta})=\int_{\eta_{\delta}}^1 g(u_{\delta}(s))s^2ds
    -{\nu\over3}\eta^3_{\delta} % \\
    < \int_{\eta_{\delta}}^1 g(\phi(s))s^2ds-
    \frac{1}{(\bar\sigma-\delta)^3}
    \int_{\sigma_Q}^{\bar\sigma}g(t)(t-\delta)^2dt %\\  
=0.
\end{aligned} 
\end{matrix}
\end{equation*}
This implies  $R_c<  R_s < R_\delta$,  and consequently, \eqref{2.31} holds.
% $\eta_s\le \eta_\delta$.
Similarly, we have
\begin{equation*}
  \bar\sigma\ge\sigma_Q+\frac{1}{6}R^2_{\delta}f(\sigma_Q)(1-\eta_{\delta})^2(1+2\eta_{\delta})
\ge\sigma_Q+\frac{1}{6}R^2_{\delta}f(\sigma_Q)(1-\eta^2_{\delta})
{\bar\sigma-\sigma_Q\over \bar\sigma-\delta}.
\end{equation*}   
Thus the estimate (\ref{2.32}) follows.    
\end{proof}

\medskip

% {\bf Remark 2.8} \   
\begin{rem}
By the convexity, we can also give an improved lower bound for $\sigma^*$. 
Since $g$ is strictly increasing and $g(\sigma_Q)<g(\tilde\sigma)=0$, the following function 
\begin{equation}
    G(\bar\sigma):= \int_{\sigma_Q}^{\bar\sigma} g(t)(t-\sigma_Q)^2 dt \qquad\mbox{for}\;\; \bar\sigma>\sigma_Q,
\end{equation}
  has a unique root denoted by $\bar\sigma_g$.  
  Clearly, $\bar\sigma_g>\tilde\sigma$. 

 Let $\psi(s):=\sigma_Q+(\bar\sigma_g-\sigma_Q)s$
 for $0\le s\le 1$. By the convexity of $Y(s,\bar\sigma_g)$ in $s$, we have $Y(s,\bar\sigma_g)< \psi(s)$ for $0<s<1$. Then
$$
  \mathcal F(\bar\sigma_g)=\int_0^1 g(Y(s,\bar\sigma_g))s^2ds< \int_0^1 g(\psi(s))s^2ds
  ={1\over (\bar\sigma_g-\sigma_Q)^3}G(\bar\sigma_g)=0=\mathcal F(\sigma^*).
$$
  It readily follows from $\mathcal F'(\bar\sigma)>0$ that  $\sigma^*>\bar\sigma_g$.    
\end{rem}

%{\color{red} Remark 2.8 \;\;   Give some examples of $g$ to verify
%Propositions 2.6 and 2.7.
%}

 \medskip
 \hskip 1em

\section{Asymptotic behavior}\label{3}
 \setcounter{equation}{0}
\hskip 1em

  In this section, we study the asymptotic behavior of the transient solution of 
  free boundary problem (\ref{1.1}). We first show the 
  global existence and uniqueness of the solution.
  
\medskip

  % {\bf Theorem 3.1} \ \ {\em  
\begin{thm}\label{thm3.1}
    Suppose $(\rm A1)$--$(\rm A3)$ and $\bar\sigma>\sigma_0$
  hold. For any $R_0>0$,  problem $(\ref{1.1})$ 
  has a unique global-in-time transient solution $(\sigma(r,t), R(t))$ which exists for $0\le r\le R(t)$ and $t\ge 0$. Moreover,
  \begin{equation}\label{3.4}
  R_0{\rm e}^{-\nu t/3}\le R(t)\le R_0{\rm e}^{g(\bar\sigma)t/3}
  \qquad\mbox{for}\;\;t>0.
  \end{equation}
\end{thm}

% \medskip
 
 % {\bf Proof}. \   

 \begin{proof}
     
 $(i)$ If $\sigma_0<\bar\sigma\le \sigma_Q$, the first two lines of 
 problem $(\ref{1.1})$ can be rewritten as     
\begin{equation}\label{3.1}
\left\{
\begin{array}{l}
\Delta_r\sigma=f(\sigma)  \qquad\mbox{for}\;\; 0<r<R(t),\;\;t>0,

\\[0.3cm]
\sigma_r(0,t)=0, \;\; \sigma(R(t),t)=\bar\sigma  \qquad\mbox{for}\;\;t>0,

\end{array}
\right.
\end{equation}
By a comparison principle, we have 
$$
  \sigma_0<\sigma(r,t)\le\bar\sigma\le\sigma_Q \qquad
  \mbox{for}\;\; 0\le r\le R(t),\;\;t\ge0.
$$
The above relation enables us to reduce problem (\ref{1.1}) into the following problem
 \begin{equation}\label{3.2}
   {d R(t)\over dt}=-{\nu\over3} R(t)\quad\mbox{for}\;\;t>0,  \qquad R(0)=R_0.
 \end{equation}
Obviously,  $R(t)=R_0{\rm e}^{-\nu t/3}$ and $\displaystyle\lim_{t\to+\infty}R(t)=0$.

   $(ii)$ If $\bar\sigma>\sigma_Q$, we know from Lemmas \ref{lem2.1}--\ref{lem2.3} that    
   problem $(\ref{1.1})$ 
   is equivalent to the following Cauchy problem 
\begin{equation}\label{3.3}
   {d R(t)\over dt}=R(t)F(R(t))\quad\mbox{for}\;\;t>0,  \qquad R(0)=R_0,
\end{equation}
where $F(R)$ is defined by (\ref{2.10}) 
   and (\ref{HR}). 
%(3.3)
  Since $ -{1\over3}\nu  \le F(R)\le {1\over3}g(\bar\sigma)$,
  we readily see that problem (\ref{3.3}) has a unique global solution $R(t)$
   for any $R_0>0$, and moreover, \eqref{3.4} holds. 
\end{proof}
% \hfill$\Box$
  
% \medskip

From (\ref{3.4}) and Theorem \ref{thm2.5} $(iii)$, we see that if 
$\bar\sigma\le \tilde\sigma$, then for any $R_0>0$, the corresponding solution 
$R(t)$ finally vanishes as $t$ goes to $+\infty$.

For $\bar\sigma>\tilde\sigma$, we know from $(\ref{2.13})$ and $(\ref{2.19})$
that $F'(R)<0$ for all $R>0$ and from Theorem \ref{thm2.5} $(i)$--$(ii)$ that $R_s$ uniquely exists. Then 
the classical
ODE theory tells us that, for any $R_0>0$, 
the corresponding solution $R(t)$
will converge to $R_s$ exponentially fast.

In conclusion, we have the following theorem:
\medskip

  % {\bf Theorem 3.2} \ \ {\em 
\begin{thm}\label{thm3.2} The dynamical behavior for problem \eqref{1.1} is as follows:
 
  $(i)$ If $\sigma_0<\bar\sigma\le \tilde\sigma$, 
  then for any $R_0>0$, $\displaystyle \lim_{t\to+\infty}R(t)=0$, 
  and the tumor will finally disappear. 
    
  $(ii)$ If $\tilde\sigma<\bar\sigma\le\sigma^*$, then for any $R_0>0$, 
   $\displaystyle\lim_{t\to+\infty}R(t)=R_s$, and the tumor will 
   finally converge to the proliferating stationary tumor with
    radius $R_s$ containing only proliferating cells.
   
   $(iii)$ If  $\bar\sigma>\sigma^*$, then for any $R_0>0$, 
   $\displaystyle\lim_{t\to+\infty}R(t)=R_s$ and 
   $\displaystyle\lim_{t\to+\infty}\rho(t):=\lim_{t\to+\infty}\rho(R(t)) =\rho_s$, 
   and the tumor will finally converge to the 
   proliferating-quiescent stationary tumor with radius $R_s$, 
   containing a quiescent core with radius $\rho_s$.   
\end{thm}

% \medskip

In this model, tumor has three growth state: $(1)$ proliferating one-layer state
with only proliferating cells; $(2)$ quiescent one-layer state with only quiescent 
cells;  $(3)$ proliferating-quiescent two-layer state with a quiescent core. 
As is shown in Lemma \ref{lem2.3}, if $\bar\sigma>\sigma_Q$, there exists a positive 
critical radius $R_c$, depending on $\bar\sigma$, such that the tumor stays in the proliferating state for
 $0<R\le R_c$,  and in the proliferating-quiescent two-layer state for 
 $R> R_c$. Obviously, tumor is always in the quiescent one-layer 
 state for $\bar\sigma\le\sigma_Q$. 
 
As we just said, if $\bar\sigma>\sigma_Q$, the tumor can be either in the proliferating one-layer state, or in the proliferating-quiescent two-layer state. These two states can be mutually transformed during the evolution process of tumor, in which we can observe the formation and vanishing of the quiescent core. More precisely, we have
 
\medskip

% {\bf  Corollary 3.3} \ \ {\em 
\begin{cor} For problem \eqref{1.1}, there exists a mutual transformation between the one- and two-layer state in the following sense: 

$(i)$ If $\sigma_Q<\bar\sigma<\sigma^*$,
  then for any initial radius $R_0>R_c$,  there exists a finite time $T>0$, 
  such that the tumor is in proliferating-quiescent two-layer state for 
  $0<t<T$, 
  and in proliferating one-layer state for $t\ge T$.
  
  $(ii)$ If $\bar\sigma>\sigma^*$, then for any initial radius $R_0<R_c$, 
  there exists a finite time $T>0$ such that the tumor is in  
  proliferating one-layer state for $0<t<T$, and
  in proliferating-quiescent two-layer state  for  $ t\ge T$.     
\end{cor}

% \medskip

% {\bf Proof}. \  

\begin{proof}
  Suppose $\sigma_Q<\bar\sigma<\sigma^*$
  and  $R_0>R_c$. By Lemma \ref{lem2.3} we see at time $t=0$, the tumor
  has a quiescent core with radius $\rho_0=\rho(R_0)$ and 
  $F(R_0)<0$. By (\ref{2.13}), (\ref{2.19}) and Theorem \ref{thm2.5},  $R(t)$
  is strictly decreasing and converges to $R_s\in [0,R_c)$ as $t$ goes to infinity. 
  Hence there
  exists a time $T>0$ such that $R(T)=R_c$ with the properties that $R(t)>R_c$ 
  for $t<T$ and $R(t)<R_c$ for $t>T$.  The assertion $(i)$ follows.  
  The assertion $(ii)$ can be proved in a similar way. 
  % The proof is complete.    
\end{proof}
  % \hfill$\Box$

\bigskip

\section{Connections with the necrotic core}\label{4}
 \setcounter{equation}{0}
\hskip 1em

In this section, we show the connections between the tumor model
with a quiescent core and the tumor model with a necrotic core.

For simplicity of notation, we take $\sigma_Q=\hat\sigma$, where $\hat\sigma$ is regarded as a threshold value of nutrient concentration for cell necrosis, cf. \cite{cui-06}, and with this, the necrotic region can be expressed by 
$\{\sigma\le \hat\sigma\}$.
 In the case $h(\sigma) \equiv 0$, problem (\ref{1.1}) becomes the 
 following necrotic tumor model:
\begin{equation}\label{4.1}
\left\{
\begin{array}{l}
  \Delta_r \sigma = f(\sigma) H(\sigma-\hat\sigma)\qquad\mbox{for}\;\; 0<r<R(t),\;\;t>0,
\\ [0.3 cm]
  \sigma_r(0,t)=0,\quad \sigma(R(t),t)=\bar\sigma \qquad \mbox{for}\;\;t>0,
\\[0.3 cm]
  \displaystyle
  R^2(t){dR(t)\over dt}=\int_{\sigma(r,t)>\hat\sigma} g(\sigma(r,t))r^2dr
  -\int_{\sigma(r,t)\le\hat\sigma}\nu r^2dr \qquad\mbox{for}\;\;t>0,
\\ [0.3 cm]
  R(0)=R_0.
\end{array}
\right.
\end{equation}
%(4.1)
The above model (\ref{4.1}) has been well studied in  \cite{bue-erc-08, cui-06, cui-fri-01, wu-wang-19}. 
By comparing Lemma \ref{lem2.4}, Theorem \ref{thm2.5} with \cite[Theorem 2.3]{wu-wang-19}, we find that, both tumor models share the same constants $\sigma^*\in (\hat\sigma,+\infty)$ and $R_c$ (depending on $\bar\sigma$) because $f$ and $g$ keep the same assumptions in the models. As is illustrated in \cite{wu-wang-19}, 
for $\bar\sigma>\sigma^*$, problem (\ref{4.1}) has a unique stationary
necrotic solution $(\sigma_{nec}(r), \rho_{nec}, R_{nec})$ with $R_{nec}
>\rho_{nec}>0$. 

In the rest of this paper, we always assume $\bar\sigma>\hat\sigma$.
Define 
\begin{equation}
  \|h\|:=\max_{\sigma_0\le \sigma \le \hat\sigma} {|h(\sigma)|}\qquad
  \mbox{for}\;\; h\in C[\sigma_0,\hat\sigma].
\end{equation}
  Under the condition $(\rm A1)$, we have $\|h\|= h(\hat\sigma)$.
  We rewrite the corresponding stationary radius $R_s$ of problem \eqref{1.1} with $h$ by 
  $R_s=R_{s,h}$. Clearly, problem (\ref{1.1}) formally converges 
  to problem (\ref{4.1}) if $\|h\|$ vanishes.  It is an interesting question whether 
  $R_{s, h}$ will converge to $R_{nec}$. In this section, we shall 
  give an affirmative answer.

One might have noticed from Lemma \ref{lem2.2} that, the value $\sigma'(\rho)$ plays a role in solving nutrient concentration $\sigma$ in the proliferating region, 
and it actually depends on 
 the functions $h$ and $\rho$. To clarify the dependence of
$\sigma$ on $\rho$,  we  consider the following auxiliary problem:
\begin{equation}\label{4.3}
\left\{
\begin{array}{l}
 \displaystyle\sigma''(r)+{2\over r}\sigma'(r) = f(\sigma(r))
  \qquad\mbox{for}\;\; \rho<r<R,
\\ [0.3 cm]
  \sigma'(\rho)=K\rho,
  \;\; \sigma(\rho)=\hat\sigma, 
\\[0.3 cm]
  \sigma(R)=\bar\sigma,
%  \displaystyle
%  \int_\rho^R g(\sigma(r))r^2dr-{\nu\over 3}\rho^3=0.
\end{array}
\right.
\end{equation}
  where $f(\sigma)$ satisfies $(\rm A1)$, $K\ge0$ and 
  $\bar\sigma>\hat\sigma>\sigma_0$. 
  
\bigskip

  % {\bf Lemma 4.1} \ \  {\em 
\begin{lem}\label{lem4.1}

  Suppose $(\rm A1)$ and
   $\bar\sigma>\hat\sigma>\sigma_0$ hold.
  For any given  $\rho\ge 0$ and $\displaystyle 0\le K< {1\over 3}f(\hat\sigma)$,  
   problem $(\ref{4.3})$ has a unique solution 
  $(\sigma(r,\rho, K), R(\rho, K))$,  and there hold the following assertions: 

% $(i)$  $\hat\sigma\le \sigma(r,\rho,K)\le \bar\sigma$ for $\rho\le r\le R(\rho,K)$.

 $(i)$ $\sigma(r,\rho,K)$ is strictly increasing in $r$ and strictly convex in $r$.
 
 $(ii)$ $\sigma(r, \rho, K)$ is strictly increasing in $K$,
 and $R(\rho,K)$ is strictly decreasing in $K$.

$(iii)$ $\sigma(r, \rho, K)$ is strictly decreasing in $\rho$, 
and $R(\rho,K)$ is strictly increasing in $\rho$.
%$$
%  \lim_{\phi \to 0^+}R(\phi)=R_{nec}, \qquad 
%  \lim_{\phi \to 0^+}\rho(R(\phi), \phi)=\rho_{nec}.
%$$
\end{lem}

% \medskip
% {\bf Proof}. \   

\begin{proof}

Proceeding as in Section \ref{2} where we dealt with problem \eqref{2.5} for $\rho<r<R$, we can solve the first two lines of problem 
  $(\ref{4.3})$ to get the unique solution  $\sigma(r, \rho, K)$,
  then by solving the third equation $\sigma(R, \rho, K)=\bar\sigma$ to obtain the unique $R=R(\rho, K)\in (\rho, +\infty)$. This gives the 
unique solution 
  $(\sigma(r,\rho, K), R(\rho, K))$ for $\rho\le r\le R(\rho,K)$.

$(i)$ Similarly to Lemma \ref{lem2.2}, we see $\sigma(r,\rho,K)\ge \hat\sigma$ for $r\ge \rho$,
  and $\sigma(r,\rho,K)$ is strictly increasing in $r$. It is easy to see that
$$
  \sigma'(r)={1\over r^2} \Big(K\rho^3+ \int_\rho^r f(\sigma(s))s^2 ds\Big)>0.
$$
  From the monotonicity of $f$ and $\sigma(r)\ge \hat\sigma$ for $r\le\rho$,
  we have
$$
  \sigma''(r)=f(\sigma(r))-{2\over r}\sigma'(r)
  \ge {1\over 3} f(\hat\sigma)+{2 \rho^3\over r^3}
   \Big({1\over 3} f(\hat\sigma)-K\Big)>0\qquad
   \mbox{for}\;\; r\ge\rho,
$$   which means $\sigma(r,\rho,K)$ is strictly convex in $r$.
  
  $(ii)$ For any $r>\rho$, we see that $v(r):={\partial \sigma\over \partial K}$ satisfies 
$$
\left\{
\begin{array}{l}
 \displaystyle  v''+{2\over s} v' =f'(\sigma) v \qquad \mbox{for}\;\; \rho<s<r,
  \\ [0.3 cm]
  v'(\rho)=\rho,\qquad v(\rho)=0.
\end{array}
\right.  
$$
  Since $f'(\sigma)> 0$, it follows from the strong maximum principle that
\begin{equation}
  \max\limits_{\rho\le s\le r} v(s) = v(r)>0,
\end{equation}
 i.e.,  ${\partial \sigma\over \partial K}>0$. Then from equation 
 $\sigma(R,\rho,K)=\bar\sigma$,  we get  $
  {\partial R\over\partial K}=-{\partial \sigma\over \partial K}/{\partial \sigma\over
  \partial R}<0.$

  $(iii)$ From $(4.3)$ we see that $w(r)={\partial \sigma\over \partial \rho}$ satisfies 

$$  \left\{
\begin{array}{l}
 w''(s)+{2\over s} w'(s)=f'(\sigma(s))w(s) \qquad \mbox{for}\;\; \rho<s<r,
  \\ [0.3 cm]
  w'(\rho)=K-\sigma''(\rho)=3K-f(\hat\sigma)<0, \\[0.3cm]
 w(\rho)=-\sigma'(\rho)  =-K\rho\le 0.  
\end{array}
\right.$$  
Again, it follows from the strong maximum principle that 
\begin{equation}
  \min_{\rho\le s\le r}w(s)=w(r)<0,\qquad\mbox{for}\;\;
  r>\rho\ge 0.
\end{equation}
  This implies ${\partial \sigma\over\partial \rho}<0$ and, in a similar manner, we obtain
  ${\partial R\over \partial \rho}>0$ for $\rho\ge 0$.    
\end{proof}
  
% \bigskip

From Lemma \ref{lem4.1}, $\sigma(r,\rho,K)$ and $R(\rho,K)$ is continuous and we have
\begin{equation}\label{4.6}
  \lim_{K\to 0^+} \sigma(r,\rho,K)=\sigma(r,\rho, 0),\qquad
  \lim_{K\to 0^+} R(\rho,K)=R(\rho, 0),
\end{equation}
in which, due to the monotonicity,   
 $(\sigma(r,\rho,0), R(\rho,0))$ satisfies problem \eqref{4.3} for $\rho\ge 0$ and 
 $K=0$.
 
  By comparing problem \eqref{2.5} and problem \eqref{4.3},  from definition \eqref{Rc} we see
\begin{equation}\label{4.7}
  R_c=\lim_{\rho\to 0^+} R(\rho, K)\qquad\mbox{for all}\;\;
  0\le K<{1\over3}f(\hat\sigma).
\end{equation}
  Then from the equation $\sigma(R,\rho,K)=\bar\sigma$, we can get
  $\rho=\rho(R,K)$,
  and
\begin{equation}\label{4.8}
  {\partial \rho\over \partial R}>0,\qquad {\partial\rho\over \partial K}>0
  \qquad \mbox{for}\;\; R\ge R_c,\; 0\le K<{1\over 3}f(\hat\sigma).
\end{equation}
  Similarly, we have
\begin{equation}\label{4.9}
  \lim_{K \to 0^+}\rho(R,K)=\rho(R,0),\qquad
  \rho(R_c, K)=0.
\end{equation}
  Next, we let
\begin{equation}\label{4.10}
\begin{matrix}
\begin{aligned}
   F(R, K)&=\frac{1}{R^3}\Big(\int_{\rho(R, K)}^{R}g(\sigma(r,\rho(R,K),K))r^2dr
  -\frac{\nu}{3}\rho^3(R, K)\Big) \\
&=\int_{\eta(R, K)}^{1}g(\mathcal U(s,R, K))s^2ds-\frac{\nu}{3}(\eta(R,K))^3,
\end{aligned}
\end{matrix}
\end{equation}
where 
$$
\mathcal U(s,R, K)=\sigma(sR,\rho(R,K),K), \qquad \eta(R,K)=\rho(R, K)/R.
$$
  By (\ref{2.16}) and (\ref{4.9}),  for $\bar\sigma>\sigma^*$ and
  $\displaystyle 0\le K<{1\over 3}f(\hat\sigma)$,  
\begin{equation}\label{4.11}
  {\partial F(R_c, K)\over \partial K}=0\qquad
  \mbox{and}\qquad 
  F(R_c, K)>0. 
\end{equation}
Moreover, we have the following result.
\bigskip

  % {\bf Lemma 4.2} \ \  {\em 

\begin{lem}\label{lem4.2}

  Suppose $(\rm A1)$--$(\rm A3)$ and
   $\bar\sigma>\sigma^*$ hold.  For $R>R_c$ and 
   $\displaystyle 0\le K<{1\over 3}f(\hat\sigma)$, there hold
\begin{equation}\label{4.12}
  {\partial F(R, K)\over \partial K}<0\qquad
  \mbox{and}\qquad 
   {\partial F(R, K)\over \partial R}<0. 
\end{equation}
    
\end{lem}

% \medskip

% {\bf Proof.} \  
\begin{proof}
Since $\sigma(r,\rho,K)$ satisfies problem \eqref{4.3} for 
 $\rho=\rho(R,K)$ and $R\ge R_c$,  
 $\mathcal U(s,R, K)$ is the solution of the following problem 
\begin{equation}\label{4.13}
\left\{
\begin{array}{l}
 \mathcal U_{ss}(s,R, K)+\displaystyle\frac{2}{s}\mathcal U_s(s,R, K)=R^2f(\mathcal U(s,R,K)) 
 \qquad\mbox{for}\;\; \eta<s<1,
\\ [0.3 cm]
  \mathcal U(1,R, K)=\bar\sigma, \;\; 
  \mathcal U(\eta,R,K)=\hat\sigma, \;\;  
  \mathcal U_s(\eta,R, K)=KR^2\eta,
\end{array}
\right.
\end{equation}
  where $\eta=\eta(R,K)$. Let
$$
  v(s)=\frac{\partial\mathcal U}{\partial K}(s,R, K), \qquad \zeta=\frac{\partial\eta}{\partial K}(R,K),
$$
then $v(s)$ satisfies
\begin{equation}\label{4.14}
\left\{
\begin{array}{l}
 v''(s)+\displaystyle\frac{2}{s}v'(s)
 =R^2f'(\mathcal U)v \qquad\mbox{for}\;\; \eta<s<1,
\\ [0.3 cm]
  v(1)=0, \quad  v(\eta)=-KR^2\eta\zeta,
\\ [0.3 cm]
  v'(\eta)=R^2\eta+R^2\zeta(3K-f(\hat\sigma)).
\end{array}
\right.
\end{equation}
  Since $3K-f(\hat\sigma)<0$, similarly to \eqref{2.9}, by the strong maximum 
  principle we obtain
\begin{equation}\label{4.15}
  v(s)={\partial\mathcal U\over \partial K}< 0\quad \mbox{for}\;\; \eta<s<1,
  \qquad\mbox{and}\qquad
  \zeta={\partial \eta\over \partial K}>0.
\end{equation}
%(4.15)
  Hence we have
\begin{equation}\label{4.16}
{\partial F\over \partial K}=\int_{\eta(R,K)}^{1}
g'(\mathcal U(s,R,K))\displaystyle
  \frac{\partial\mathcal U}{\partial K}
  (s,R,K)s^2ds-(g(\hat\sigma)+\nu)
  \eta^2(R,K) {\partial\eta\over \partial K}(R,K)<0.
\end{equation}
  Similarly, we can obtain
\begin{equation}\label{4.17}
{\partial\mathcal U\over \partial R}(s,R,K)< 0\quad \mbox{for}\;\; \eta<s<1
  \qquad\mbox{and}\qquad
  {\partial \eta\over \partial R}>0.
\end{equation}
  It follows that
\begin{equation}\label{4.18}
{\partial F\over \partial R}=\int_{\eta(R,K)}^{1}g'(\mathcal U(s,R,K))
\displaystyle
  \frac{\partial\mathcal U}{\partial R}(s,R,K)s^2ds-(g(\hat\sigma)+\nu)
  \eta^2(R,K) {\partial\eta\over \partial R}(R,K)<0.
\end{equation}
This completes the proof. \end{proof} % \hfill$\Box$

% \bigskip

Similarly to Lemma \ref{lem2.3} $(iii)$  and (\ref{2.12}), we can easily verify that
for any $\displaystyle 0\le K<{1\over 3}f(\hat\sigma)$,  
\begin{equation}\label{4.19}
  \lim_{R\to +\infty} \eta(R,K)=1\qquad \mbox{and}\qquad
  \lim_{R\to +\infty} F(R,K)=-{\nu\over 3}<0.
\end{equation}
  By (\ref{4.11}) and Lemma \ref{lem4.2},  we see the equation $F(R,K)=0$ has
  a unique root $R=R_s(K)>R_c$ for any 
  $\displaystyle 0\le K<{1\over 3}f(\hat\sigma)$,
  and
\begin{equation}\label{4.20}
  {d R_s\over d K}<0. 
\end{equation}
  Thus by (\ref{4.6}) we immediately get 
\begin{equation}\label{4.21}
 \lim_{K\to 0^+} F(R,K)=F(R,0)\qquad \mbox{and}\qquad
  \lim_{K\to 0^+} R_s(K)=R_s(0).  
\end{equation}
  Letting $\rho_s(K)=\rho(R_s(K), K)$ and $\sigma_s(K)=\sigma(r,\rho_s(K), K)$,
  we see $(\sigma_s(0), \rho_s(0), R_s(0))$ is the stationary solution of
  problem (\ref{4.1}),  i.e.,
$$
(\sigma_s(0), \rho_s(0),R_s(0))=(\sigma_{nec}(r), \rho_{nec}, R_{nec}).
$$
  Hence we obtain
\begin{equation}\label{4.22}
  \lim_{K\to 0^+}\sigma(r, \rho_s(K), K)=\sigma_{nec}(r),\quad
  \lim_{K\to 0^+}\rho_s(K)=\rho_{nec},\quad
  \lim_{K\to 0^+}R_s(K)=R_{nec}.
\end{equation}

 \medskip
 
  Let $(\rm A1)$--$(\rm A3)$ hold and $f, g$ be fixed. Denote the corresponding stationary solution of problem \eqref{1.1} with $h$ by
 $(\sigma_{s,h}(r), \rho_{s,h}, R_{s,h})$. 
 With the above preparations, we present the following 
 connection between the tumor models with the quiescent core 
 and with the necrotic core. 
 
\bigskip

  % {\bf Theorem 4.3} \ \  {\em 

\begin{thm}\label{thm4.3}
  Suppose $(\rm A1)$--$(\rm A3)$ and 
  $\bar\sigma>\sigma^*$ hold. Then the unique stationary solution of 
  proliferating-quiescent tumor model \eqref{1.1}
   converges to the stationary solution of the necrotic tumor model \eqref{4.1}, 
   as $\|h\|$ goes to $0$. More precisely, we have
%$$
%\mathscr{A}=\left\{h\in C^1_{[\sigma_0,\hat\sigma]} \vert h 
%\;is\;strictly\;increasing\; and\; h(\sigma_0)=0, h(\hat\sigma)\le f(\hat\sigma) 
%\right\},
%$$
%  then when $\bar\sigma>\sigma^*$, for each $h\in\mathscr{A}$, there exists a 
%unique stationary solution with dormant core of problem $(\ref{4.2})$ given by a 
%triple $(\sigma(r,h),\rho(h),R(h))$.
%By denoting 
$$
\lim_{\|h\| \to 0^+}\sigma_{s,h}(r)=\sigma_{nec}(r),
\quad  \lim_{\|h\| \to 0^+}\rho_{s,h}=\rho_{nec},
\quad \lim_{\|h\| \to 0^+}R_{s,h}=R_{nec}.
$$

\end{thm}

% \medskip

% {\bf Proof.} \ 

\begin{proof}

Recall $\hat\sigma=\sigma_Q$. For any $h$ satisfying $(\rm A1)$, we have
$$
\sigma'(\rho)=\Phi(\rho)={1\over \rho^2}\int_0^{\rho} s^2h(V(s,\rho))ds\le\frac{1}{3}
  \rho h(\hat\sigma)={1\over 3} \rho\|h\|.
$$
  From  (\ref{4.20}), we see that for any $\displaystyle {1\over 3}h(\hat\sigma)< 
  K< {1\over 3}f(\hat\sigma)$, 
$$
  R_{nec}=R_s(0)>R_{s,h}=R_s({\Phi(\rho)\over \rho})>R_s(K).
$$
  It follows that 
$$
  \lim_{\|h\|\to 0^+} R_{s,h}=R_{nec}.
$$
  Then similarly to (\ref{4.22}), we readily get
$$
\lim_{\|h\| \to 0^+}\sigma_{s,h}(r)=\sigma_{nec}(r)
\qquad \mbox{and}\qquad  \lim_{\|h\| \to 0^+}\rho_{s,h}=\rho_{nec}.
$$
This completes the proof. \end{proof} %\hfill $\Box$

With the aid of Lemmas \ref{lem4.1}--\ref{lem4.2} and the monotonicity of $F$, 
only slight modifications in the above deductions are required to prove a more general result that not only 
the stationary solution but also the transient  
solution of problem \eqref{1.1} will converge to 
the corresponding solution of problem \eqref{4.1} as $\|h\|$ goes to $0$.
This indicates that the tumor model with a quiescent core not only converges to the tumor model with a necrotic core in the structure, but they also  possess a vital link in their solution properties. We omit the details here.
  
\bigskip

% {\bf Remark 4.3} \ 

\begin{rem}
In the case where 
$f$, $h$, $g$ are all linear functions with the form
\begin{equation}\label{fhg}
f(\sigma)=\sigma,\qquad h(\sigma)=\lambda\sigma, \qquad  g(\sigma)=\mu(\sigma-\tilde\sigma),
\end{equation}
 problem (\ref{1.1}) has been well studied by Cui \cite{cui-06} for $\lambda=0$, and by 
 Liu and Zhuang \cite{liu-zhuang} for 
 $0<\lambda\le 1$.
 In this case, $\|h\|$ goes to $0$ if and only if $\lambda$ goes to $0$.
 We redenote the steady tumor radius $R_s$ by $R_{s,\lambda}$ for 
 $0\le \lambda\le 1$.
According to the results in \cite{cui-06} and \cite{liu-zhuang}, 
for $\bar\sigma>\sigma^*(>\tilde\sigma>\hat\sigma)$,
we can obtain the unique root $R_c>0$ of the following equation
$$
\frac{\sinh R}{R}=\frac{\bar\sigma}{\hat\sigma}.
$$
  The nutrient concentration in tumor region is given by
\begin{equation*}%\label{4.24}
  \sigma(r,R,\lambda)=\left\{
  \begin{array}{l}
  \displaystyle\frac{\rho\hat\sigma}{\sinh\sqrt{\lambda}\rho}\frac{\sinh\sqrt{\lambda}r}{r}\qquad\mbox{for}\;\;0<r<\rho,
  \\ [0.4 cm]
 \displaystyle\frac{\hat\sigma}{r}\big[\sqrt{\lambda}\rho  \coth\sqrt{\lambda}\rho  
 \sinh(r-\rho)+\rho \cosh(r-\rho)\big]\qquad\mbox{for}\;\; \rho<r<R,
 \end{array}
 \right.
\end{equation*}
where $\rho=\rho(R,\lambda)$ satisfies the equation
\begin{equation*}%\label{4.25}
L(\rho,R,\lambda):= \sqrt{\lambda}\rho \coth\sqrt{\lambda}\rho  \sinh(R-\rho)+\rho \cosh(R-\rho)-\frac{\bar\sigma}{\hat\sigma}R=0,
\end{equation*}
for $R>R_c$ and $0<\lambda\le 1$.
%$$
%\phi'(\rho)=[\sqrt{\lambda} \coth\sqrt{\lambda}\rho-
%\lambda\rho\;csch^2\sqrt{\lambda}\rho ]\sinh(R-\rho)+(1-\sqrt{\lambda}\rho 
%\coth\sqrt{\lambda}\rho)\cosh(R-\rho)<0
%$$
One can easily verify that
\begin{equation}\label{dandiao}
{\partial L\over \partial \rho}<0,\qquad
{\partial L\over \partial \lambda}>0,\qquad
{\partial\rho\over \partial  \lambda}=
-{\partial_{\lambda}L \over \partial_{\rho} L}>0.
%=-\frac{\rho \sinh(R-\rho)}{2\sqrt{\lambda}\phi'(\rho)}[\coth\sqrt{\lambda}\rho-
%\sqrt{\lambda}\rho\;csch^2\sqrt{\lambda}\rho]
\end{equation} 
  The steady tumor radius $R_{s,\lambda}$ is 
  the unique root of the following equation
\begin{equation*}%\label{4.27}
F(R,\lambda):={1\over R^3}\Big[\int_{\rho(R,\lambda)}^{R}\mu(\sigma(r,R,\lambda)-\tilde\sigma)r^2dr-\frac{\nu}{3}\rho^3(R,\lambda)\Big]=0.
\end{equation*}
 Similarly, by \cite{liu-zhuang} and (\ref{dandiao}) we have
 for $R>R_c$ and $0<\lambda<1$,
\begin{equation}\label{dandiao2}
 {\partial F\over \partial R}<0,\qquad {\partial F\over \partial \lambda}<0,
 \qquad
 {d R_{s,\lambda}\over d \lambda}=- {\partial_\lambda F
 \over\partial_R F}<0.
\end{equation}
  On the other hand,
 for the necrotic tumor model with
$f, h, g$ given by (\ref{fhg}) and $\lambda=0$, 
from \cite{cui-06}, we know that 
\begin{equation*}%\label{4.29}
  \sigma(r,R,0)=\left\{
  \begin{array}{l}
  \hat\sigma \qquad\mbox{for}\;\;0<r<\rho,
  \\ [0.4 cm]
 \displaystyle\frac{\hat\sigma}{r}\big[ \sinh(r-\rho)+\rho 
 \cosh(r-\rho)\big]\qquad\mbox{for}\;\; \rho<r<R.
 \end{array}
 \right.
\end{equation*}
where $\rho=\rho(R,0)$ satisfies
\begin{equation*}%\label{4.30}
  L(\rho,R,0):=\sinh(R-\rho)+\rho \cosh(R-\rho)-\frac{\bar\sigma}{\hat\sigma}R=0,
\end{equation*} 
  and $R_{s,0}$ is the unique root of the following equation:
\begin{equation*}
F(R, 0):={1\over R^3}\Big[\int_{\rho(R,0)}^{R}\mu(\sigma(r,R,0)-\tilde\sigma)r^2dr-\frac{\nu}{3}\rho^3(R,0)\Big]=0.
\end{equation*}  
  By a direct computation, we can verify that for $R>R_c$, 
\begin{equation*}%\label{4.31}
\lim_{\lambda \to 0^+}\rho(R, \lambda)=\rho(R,0),\quad
\lim_{\lambda \to 0^+}\sigma(r,R,\lambda)=\sigma(r,R,0),\quad
\lim_{\lambda \to 0^+} F(R,\lambda)=F(R,0).
\end{equation*}
  Hence from (\ref{dandiao2}) we obtain
 $$\displaystyle\lim_{\lambda \to 0^+}\rho_{s,\lambda}=\rho_{s,0}
 =\rho_{nec}\qquad {\rm and}\qquad
 \displaystyle \lim_{\lambda \to 0^+} R_{s,\lambda}=R_{s,0}=R_{nec}.$$
 From this, we observe that, in the linear case, the stationary solution of problem \eqref{1.1} converges
 to the stationary solution of problem \eqref{4.1}, as $\|h\|$ goes to $0$. 

 % In this linear functions case, by the above direct computations,
 % we easily get the stationary solution of problem (1.1) converges
 % to the stationary solution of problem (4.1), as $\|h\|$ goes to $0$. 
\end{rem}

% \bigskip
\vspace{0.25em}

% {\noindent \bf Acknowledgments.} \, 
\begin{ac*}
This research is supported by the National Natural Science Foundation of China under the Grant No. 12101260, 12171349 and 12271389.
\end{ac*}

% \bigskip

% \vsss

\end{CJK}

{\small
\begin{thebibliography}{99}

\bibitem{bue-erc-08}  H. Bueno, G. Ercole and A. Zumpano,
  Stationary solutions of a model for the growth of tumors and 
  a connection between the nonnecrotic and necrotic phases, 
  {\em SIAM J. Appl. Math.},
  {\bf 68} (2008), 1004--1025.
  

 \bibitem{byr-cha-96}  H. Byrne and M. Chaplain, Growth of necrotic tumors
 in the presence and absence of inhibitors, {\em Math. Biosci.}, {\bf 135}
 (1996), 187--216.
 
\bibitem{byr-cha-97}  H. Byrne and M. Chaplain, Free boundary problems
  associated with the growth and development of multicellular spheroids,
  {\em Euro. Jnl of Applied Mathematics}, {\bf 8} (1997), 639--658.

  
\bibitem{cui-05} S. Cui, Analysis of a free boundary problem modeling 
  tumor growth,  {\em Acta Math. Sinica}, English Series, {\bf 21} (2005),
   1071--1082.
  

\bibitem{cui-06}  S. Cui, Formation of necrotic cores in the growth of 
  tumors: analytic results, {\em Acta Math. Scientia}, {\bf 26} (2006), 781--796.
  
\bibitem{cui-09} S. Cui, Lie group action and stability analysis of stationary   
  solutions for a free boundary problem modeling tumor growth, {\em J. Differential
   Equations},  {\bf 246} (2009), 1845--1882.

\bibitem{cui-esc-07}  S. Cui and J. Escher,  Bifurcation analysis of an
  elliptic free boundary problem modeling stationary growth of
  avascular tumors, {\em SIAM J. Math. Anal.}, {\bf 39} (2007), 210--235.
  
\bibitem{cui-fri-01}  S. Cui and A. Friedman, Analysis of a mathematical
  model of the growth of necrotic tumors, {\em J. Math. Anal. Appl.}, {\bf 255}
  (2001), 636--677.
  
   
\bibitem{esc-mat-10} J. Escher and A-V. Matioc, Radially symmetric 
  growth of nonnecrotic tumors, {\em Nonlinear Differ. Equ. Appl.}, {\bf 17} (2010), 
  1--20.

\bibitem{fri-hu-06}  A. Friedman and B. Hu, Asymptotic stability for a free
  boundary problem arising in a tumor model, {\em J. Differential Equations},
  {\bf 227} (2006), 598--639.


\bibitem{fri-hu-08} A. Friedman and B. Hu, Stability and instability of
  Liapunov-Schmidt and Hopf bifurcation for a free boundary problem arising
  in a tumor model, {\em Trans. Amer. Math. Soc.}, {\bf 360} (2008), 5291--5342.
  
\bibitem{fri-rei-99}   A. Friedman and F. Reitich, Analysis of a mathematical
 model  for the growth of tumors, {\em J. Math. Biol.}, {\bf 38} (1999), 262--284. 
 
\bibitem{he-xing-hu}  W. He, R. Xing and B. Hu, Linear stability  analysis
  for a free boundary problem  modeling multilayer tumor growth with time delay,
  {\em SIAM J. Math. Anal.}, {\bf 54} (2022), 4238--4276.  
  
\bibitem{hua-hu-24} Y. Huang and B. Hu, Periodic solution for a
  free-boundary tumor model with small diffusion-to-growth ratio,
  {\em J. Differential Equations}, {\bf 399} (2024), 252--280.

\bibitem{hua-zha-hu-19}  Y. Huang, Z. Zhang and B. Hu, 
  Linear stability for a free boundary tumor model with a periodic supply 
  of external nutrients, {\em Math. Methods Appl. Sci.},  {\bf 42} (2019), 1039--1054.



\bibitem{hao-12}  W. Hao, J. D. Hauenstein, B. Hu and et al, Bifurcation for
  a free boundary problem modeling the growth of a tumor with a necrotic core,
  {\em Nonlinear Anal. Real World Appl.}, {\bf 13} (2012),  694--709.

\bibitem{liu-zhuang}  Y. Liu and Y. Zhuang, Analytic results of a 
    double-layered radial tumor model with different consumption rates,
    {\em Nonlinear Anal. Real World Appl.}, {\bf 76} (2024),  104004.

\bibitem{liu-z-2} Y. Liu and Y. Zhuang, Time delays in a double-layered radial tumor model with different living cells, {\em Math. Methods Appl. Sci.}, {\bf 48} (2025),  2655–-2664.

\bibitem{lu-hao-hu}  M. Lu, W. Hao, B. Hu and S. Li, Bifurcation analysis
of a free boundary model of vascular tumor growth with a necrotic core and chemotaxis, {\em J. Math. Biol.}, {\bf 86} (2023), 19.

\bibitem{wu-16} J. Wu, Stationary solutions of a free boundary problem
    modeling the growth of tumors with Gibbs-Thomson relation,
   {\em J. Differential Equations}, {\bf 260} (2016), 5875--5893.

%\bibitem{}\label{wu-17} J. Wu, Analysis of a mathematical model for 
%  tumor growth with Gibbs-Thomson relation, {\em J. Math. Anal. Appl.},
%  {\bf 450} (2017), 532--543.

\bibitem{wu-18} J. Wu, Asymptotic stability of a free boundary problem for
  the growth of multi-layer tumours in the necrotic phase, {\em Nonlinearity},
  {\bf 32} (2019), 2955--2974.


\bibitem{wu-19} J. Wu,  Bifurcation for a free boundary problem modeling
  the growth of necrotic multilayered tumors, {\em Discrete Contin. Dyn. Syst.}, 
  {\bf 39} (2019),  3399--3411. 
  
\bibitem{wu-21} J. Wu, Analysis of a nonlinear necrotic tumor model with
  two free boundaries, {\em J. Dynam. Differential Equations}, {\bf 33} (2021),
  511-524.

\bibitem{wu-wang-19} J. Wu, C. Wang, Radially symmetric growth of necrotic tumors and connection with nonnecrotic tumors, {\em Nonlinear Anal. Real World Appl.},   {\bf 50} (2019), 25-33. 

\bibitem{wu-xu-20} J. Wu, S. Xu , Asymptotic behavior of a nonlinear 
  necrotic tumor model with a periodic external nutrient supply, 
  {\em Discrete Contin. Dyn. Syst. Ser. B}, 
  {\bf 25 (7)} (2020),  2453-2460. 
  
\bibitem{xu-zhang-zhou} S. Xu, F. Zhang and Q. Zhou, A free boundary 
  problem for necrotic tumor growth with angiogenesis, {\em Appl. Anal.}, 
  {\bf 102} (2023), 977-987.
  
\bibitem{zheng-li-zhuang} J. Zheng, R. Li and Y. Zhuang, Analysis of the growth
  of a radial tumor with triple-layered structure, {\em Discrete Contin. Dyn. Syst.}, 
  {\bf 44} (2024), 1958--1981.

\bibitem{zhao-hu} X. Zhao and B. Hu, Symmetry-breaking
 bifurcation for a free-boundary tumor model with time delay, {\em J. Differential
 Equations},   {\bf 269} (2020), 1829--1862.
 
\bibitem{zhuang-cui} Y. Zhuang and S. Cui,  Analysis of a free boundary 
  problem modeling the growth of multicell spheroids with angiogenesis, 
  {\em J. Differential Equations},   {\bf 265} (2018), 620--644. 


\end{thebibliography}}
\end{document}